\documentclass[11pt,a4paper,reqno]{amsart}
\usepackage[width=0.8\paperwidth,height=0.85\paperheight,twoside=false,includehead,includefoot]{geometry}
\allowdisplaybreaks[1]

\usepackage{tikz}
\usepackage[all]{xy}

\newtheorem{thm}{Theorem}[section]
\newtheorem{lem}[thm]{Lemma}
\newtheorem{prop}[thm]{Proposition}
\newtheorem{cor}[thm]{Corollary}

\theoremstyle{definition}

\newtheorem{ex}[thm]{Example}
\theoremstyle{remark}
\newtheorem{rem}[thm]{Remark}

\DeclareMathOperator{\End}{End}

\DeclareMathOperator{\Aug}{Aug}

\newcommand{\Q}{\mathbb{Q}}

\newcommand{\C}{\mathbb{C}}
\newcommand{\I}{\mathbb{I}}

\newcommand{\A}{\mathcal{A}}
\newcommand{\calH}{\mathcal{H}}
\newcommand{\mapI}{\mathcal{I}}
\newcommand{\mapL}{\mathcal{L}}
\newcommand{\calT}{\mathcal{T}}
\newcommand{\indk}{\mathbf{k}}
\newcommand{\inds}{\mathbf{s}}
\newcommand{\indz}{\mathbf{z}}

\begin{document}

\title{Rooted tree maps for multiple $L$-values from a perspective of harmonic algebras}

\author{Hideki Murahara}
\address[Hideki Murahara]{The University of Kitakyushu, \endgraf 
4-2-1 Kitagata, Kokuraminami-ku, Kitakyushu, Fukuoka, 802-8577, Japan}
\email{hmurahara@mathformula.page}

\author{Tatsushi Tanaka}
\address[Tatsushi Tanaka]{Department of Mathematics, Faculty of Science, Kyoto Sangyo University, \endgraf 
Motoyama, Kamigamo, Kita-ku, Kyoto-city 603-8555, Japan}
\email{t.tanaka@cc.kyoto-su.ac.jp}

\author{Noriko Wakabayashi}
\address[Noriko Wakabayashi]{Center of Physics and Mathematics, Institute for Liberal Arts and Sciences, Osaka Electro-Communication University  \endgraf 
1130-70 Kiyotaki, Shijonawate-City, Osaka 575-0063 Japan}
\email{wakabayashi@osakac.ac.jp}


\begin{abstract}
In this paper, we show the image of rooted tree maps 
themselves forms a 
subspace of the kernel of the evaluation map of multiple $L$-values. For its proof, we define the diamond product as a modified harmonic product and find its properties. We also show that $\tau$-conjugate rooted tree maps are 
their antipodes.

\end{abstract}

\maketitle

\section{Introduction}\label{intro}
Let $\mu_r$ be the set of $r$th roots of unity.
For an index set $(\indk;\inds)=(k_1,\ldots ,k_l;s_1,\ldots ,s_l)$ with $k_1,\ldots, k_l\geq 1, s_1,\ldots ,s_l\in\mu_r, (k_1,s_1)\neq (1,1)$, the multiple $L$-value of shuffle type (abbreviated as MLV) is defined in \cite{AK04} by the convergent series
\[
L(\indk;\inds)=\lim_{m\to\infty}\sum_{m>m_1>\cdots >m_l>0}\frac{s_1^{m_1-m_2}\cdots s_{l-1}^{m_{l-1}-m_l}s_l^{m_l}}{m_1^{k_1}\cdots m_l^{k_l}}. 
\]
If $r=1$, this is nothing but the multiple zeta value (abbreviated as MZV). 
The MZVs and the MLVs are well-studied in the last three decades. 

The index set $(\indk;\inds)$ is often identified by the word $\indz_{\indk,\inds}:=z_{k_1,s_1}\cdots z_{k_l,s_l}$, where $z_{k,s}$ stands for $x^{k-1}y_{s}$, in the noncommutative polynomial algebra $\A_r:=\Q\langle x,y_s | s\in\mu_r \rangle$. Then MLVs are algebraically discussed via the $\Q$-linear map $\mapL:\A_r^0\to\C$ defined by $\mapL(1)=1$ and $\mapL(\indz_{\indk,\inds})=L(\indk;\inds)$. ($\A_r^0$ is a subalgebra of $\A_r$ generated by admissible words, detailed in the next section.)

On the other hand, (non-planar) rooted trees are finite and connected graphs with no cycles and a special vertex called the root. For example, 
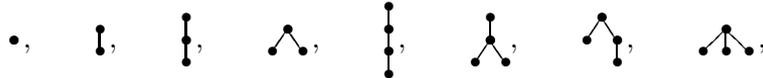
\[
\,\begin{xy}
   {(0,0) \ar @{{*}-{*}} (0,0)}, 
\end{xy}\,\,
,\qquad
\,\begin{xy}
   {(0,-1.5) \ar @{{*}-{*}} (0,1.5)}, 
\end{xy}\,\,
,\qquad
\,\begin{xy}
   {(0,-3) \ar @{{*}-{*}} (0,0)}, 
   {(0,0) \ar @{{*}-{*}} (0,3)},
\end{xy}\,\,
,\qquad
\,\begin{xy}
   {(0,-1.5) \ar @{{*}-{*}} (2,1.5)}, 
   {(2,1.5) \ar @{{*}-{*}} (4,-1.5)},
\end{xy}\,\,
,\qquad
\,\begin{xy}
   {(0,-4.5) \ar @{{*}-{*}} (0,-1.5)}, 
   {(0,-1.5) \ar @{{*}-{*}} (0,1.5)},
   {(0,1.5) \ar @{{*}-{*}} (0,4.5)},
\end{xy}\,\,
,\qquad
\,\begin{xy}
   {(0,-3) \ar @{{*}-{*}} (2,0)}, 
   {(2,0) \ar @{{*}-{*}} (4,-3)},
   {(2,0) \ar @{{*}-{*}} (2,3)},
\end{xy}\,
,\qquad
\,\begin{xy}
   {(0,0) \ar @{{*}-{*}} (2,3)}, 
   {(2,3) \ar @{{*}-{*}} (4,0)},
   {(4,0) \ar @{{*}-{*}} (4,-3)},
\end{xy}\,\,
,\qquad
\,\begin{xy}
   {(0,-1.5) \ar @{{*}-{*}} (3,1.5)}, 
   {(3,-1.5) \ar @{{*}-{*}} (3,1.5)},
   {(3,1.5) \ar @{{*}-{*}} (6,-1.5)},
\end{xy}\,\, ,
\]
and so on. The topmost vertex of each of rooted trees represents the root. The algebra generated by them has a Hopf algebra structure known as Connes-Kreimer Hopf algebra of rooted trees, appeared in \cite{CK98} in the study of perturbative quantum field theory and 
well-studied in the last quarter century. 

Rooted tree maps (abbreviated as RTMs) first defined in \cite{Tan19} based on the Connes-Kreimer Hopf algebra of rooted trees induce a certain class of relations among MZVs. In other words, a part of $\ker\mapL$ comes from the RTMs if $r=1$. Although this phenomenon is expected to be extended naturally to any positive integer $r$, the only result proved in \cite{TW22} 
is 
for RTMs taken conjugation by certain involution $\tau$. In this paper, we study some algebraic properties of RTMs for MLVs using the harmonic algebra as are studied in \cite{MT22} in the MZVs' case. 
We then show the aforementioned expectation 
is 
true and $\tau$-conjugate RTM is nothing but its antipode. 

\section{Main results}\label{results}
Let $\A_r^1$ and $\A_r^0$ be subalgebras of $\A_r$ given by
\[
A_r\supset A_r^1=\Q\oplus A_{r,+}^1 \supset A_r^0=\Q\oplus A_{r,+}^0,
\]
where 
\[
A_{r,+}^1=\bigoplus_{s\in\mu_r}\A_r y_s,\qquad A_{r,+}^0=\bigoplus_{s\in\mu_r}x\A_r y_s\oplus\bigoplus_{s,t\in\mu_r,t\neq 1}y_t \A_r y_s.
\]
Each word $\indz_{\indk,\inds}\in\A_{r,+}^0$ is called admissible and corresponds to the index set $(\indk;\inds)$ with $(k_1,s_1)\neq (1,1)$. Let $z_s^{\delta}=x+\delta(s)y_s\in\A_r$, where $\delta(1)=0$ and $\delta(s)=1$ if $s\neq 1$. 

Denote by $\calH$ the $\Q$-vector space generated by rooted forests, i.e., disjoint unions of rooted trees. This $\calH$ 
has a structure of connected Hopf algebra, 
which is briefly described in the next section. We assign to any rooted tree $t$ a linear map $\widetilde{t}\in\End_{\Q}(\A_r)$, which we call a RTM, elaborated in 
Section \ref{rtm}. The assignment $\widetilde{\ }$ is known to be an algebra homomorphism, and hence we can assign to any $f\in\calH$ a linear map $\widetilde{f}\in\End_{\Q}(\A_r)$. 
Using the notation of the diamond product $\diamond_s$ ($s\in\mu_r$), which is described in Section \ref{har}, we have the following result. 
\begin{thm} \label{main1} 
 For $f\in\calH$, there exists a unique $F_f\in\A_1^1$ such that
 \[
  \widetilde{f}(z_{s}^{\delta}w)=z_{s}^{\delta}(F_{f}\diamond_{s}w)
 \]
 holds for any $s\in\mu_{r}$ and any $w\in\A_{r}$. 
\end{thm}
%
\noindent The product $\diamond_s$ is a variation of the harmonic product. Indeed, Lemma \ref{prop0} below asserts that
\begin{align}\label{eq-1}
v\diamond_s w=\psi_s(\varphi (v) \ast \psi_s^{-1}(w)), 
\end{align}
where $v\in\A_1, w\in\A_r$, and $\ast$ is the harmonic product. Here, $\psi_s = \varphi\mapI M_s$, where $\varphi$ is the automorphism on $\A_r$ determined by $\varphi (x)=z:=x+y_1$ and $\varphi (y_s)=z_s^{\delta}-z(=\delta(s)y_s-y_1)$ for $s\in\mu_r$, and $\mapI$ and $M_s (s\in\mu_r)$ are linear maps on $\A_r$ defined by
\[
\mapI(\indz_{\indk,\inds}x^a)=z_{k_1,s_1}z_{k_2,s_1s_2}\cdots z_{k_l,s_1\cdots s_l}x^a, \qquad M_s(\indz_{\indk,\inds}x^a)=z_{k_1,ss_1}z_{k_2,s_2}\cdots z_{k_l,s_l}x^a
\]
for $a\geq 0$. Note that $\varphi$ is an involution. 
According to \cite{KT08}, we have
\begin{align}\label{eq0}
z_s^{\delta}\cdot \psi_s (\A_{1,+}^1 \ast \A_{r,+}^1)\subset\ker\mapL
\end{align}
for any $s\in\mu_r$. Hence, for $s\in\mu_r, w\in\A_{r,+}^1$, and $f\in\Aug(\calH)$, where $\Aug(\calH)$ denotes the augmentation ideal of $\calH$, i.e., $\calH=\Q\oplus\Aug(\calH)$, we have
\[
\widetilde{f}(z_s^{\delta}w)=z_s^{\delta}(F_f\diamond_s w)=z_s^{\delta}\cdot \psi_s(\varphi(F_f) \ast \psi_s^{-1}(w))\in\ker\mapL. 
\]
Thus we have the following:
\begin{cor}
 For $f\in\Aug(\calH)$, 
 we have $\widetilde{f}(\A_{r,+}^{0}) \subset\ker\mapL$. 
\end{cor}
\begin{rem}
We note that this result was expected but not proved in \cite{TW22}. Still we do not know the way to prove this directly from the definition of RTM (except that the case of $r=1$, the MZV case, is done in \cite{Tan19}). 
\end{rem}
Let $S$ be the antipode of $\calH$. Then, for $f\in\calH$, we find the antipode map $\widetilde{S(f)}$ is also described similarly by using the diamond product $\diamond_s$. 
\begin{thm} \label{main2} 
 For $f\in\calH$, there exists a unique $G_f\in\A_1^1$ such that
 \[
  \widetilde{S(f)}(z_{s}^{\delta}w)=z_{s}^{\delta}(G_{f}\diamond_{s}w)
 \]
 holds for any $s\in\mu_{r}$ and any $w\in\A_{r}$. 
\end{thm}
As is defined in \cite{TW22}, let $\tau$ be the anti-automorphism on $\A_r$ defined by $\tau(x)=y_1,\tau(y_1)=x$, and $\tau(y_s)=-y_s$ ($s\neq 1$). Note that $\tau$ is an involution. 
Then we show the following result, which is a generalization of \cite[Theorem 1.5]{MT22}. 
\begin{thm} \label{main3} 
 For $f\in\calH$, we have $\widetilde{S(f)}=\tau\widetilde{f}\tau$.
\end{thm}
\noindent Hence, for any $s\in\mu_r$. Hence, for $s\in\mu_r, w\in\A_{r,+}^1$, and $f\in\Aug(\calH)$, we have
\[
\tau \widetilde{f}\tau(z_s^{\delta}w) = \widetilde{S(f)}(z_s^{\delta}w) = z_s^{\delta}(G_f \diamond_s w) = z_s^{\delta}\cdot \psi_s(\varphi(G_f) \ast \psi_s^{-1}(w)) \in\ker\mapL 
\]
because of \eqref{eq-1}, \eqref{eq0}, and the last two theorems. Thus again we have the following proved first in \cite[Theorem 2.4]{TW22}: 
\begin{cor}
 For $f\in\Aug(\calH)$,
 we have $\tau\widetilde{f}\tau(\A_{r,+}^{0}) \subset\ker\mapL$.
\end{cor}
%

\section{Connes-Kreimer Hopf algebra of rooted trees}\label{ck}
In this section,  we briefly review the Connes-Kreimer Hopf algebra of rooted trees appeared in \cite{CK98}. A rooted tree is a finite, connected, acyclic, and oriented graph with a special vertex called the root from which every edge directly or indirectly originates. A rooted forest is a product (disjoint union) of rooted trees. The empty tree, the rooted tree of degree $0$, denoted by $\I$ is the neutral element for the product. We put $\calT$ the $\Q$-vector space freely generated by rooted trees. 

As is mentioned in the previous section, we denote by $\calH$ the $\Q$-algebra generated by rooted trees. As a vector space, $\calH$ is freely generated by rooted forests. The $\Q$-linear map called the grafting operator $B_+:\calH\to\calT$ is defined by $B_+(\I)=\,\begin{xy}{(0,0)\ar@{{*}-{*}}(0,0)}\end{xy}\,\,$  and, for a rooted forest $f$ of positive degree, all the roots of connected components of $f$ are grafted to a single new vertex, which becomes the new root. For example, we have
\[
B_{+}\left(
\,\,\begin{xy}
   {(-2.5,0) \ar @{{*}-{*}} (-2.5,0)}, 
   {(0,-1.5) \ar @{{*}-{*}} (2,1.5)}, 
   {(2,1.5) \ar @{{*}-{*}} (4,-1.5)},
\end{xy}\,\,
\right)
=
\,\begin{xy}
   {(0,-3) \ar @{{*}-{*}} (2,0)}, 
   {(2,0) \ar @{{*}-{*}} (4,-3)},
   {(2,0) \ar @{{*}-{*}} (0,3)},
   {(-2,0) \ar @{{*}-{*}} (0,3)}, 
\end{xy}\,\;,
\quad
B_{+}\left(
\,\,\begin{xy}
   {(-2.5,0) \ar @{{*}-{*}} (-2.5,0)}, 
   {(0,0) \ar @{{*}-{*}} (0,0)}, 
   {(2.5,0) \ar @{{*}-{*}} (2.5,0)},
\end{xy}\,\,
-2
\;\;\begin{xy}
   {(-3,-1.5) \ar @{{*}-{*}} (-3,1.5)}, 
   {(0,-1.5) \ar @{{*}-{*}} (0,1.5)}, 
\end{xy}\,\,
\right)
=
\,\,\begin{xy}
   {(-2.5,-1.5) \ar @{{*}-{*}} (0,1.5)}, 
   {(0,-1.5) \ar @{{*}-{*}} (0,1.5)}, 
   {(2.5,-1.5) \ar @{{*}-{*}} (0,1.5)},
\end{xy}\,\,
-2
\;\;\begin{xy}
   {(-3,-3) \ar @{{*}-{*}} (-3,0)}, 
   {(0,-3) \ar @{{*}-{*}} (0,0)}, 
   {(-3,0) \ar @{{*}-{*}} (-1.5,3)}, 
   {(0,0) \ar @{{*}-{*}} (-1.5,3)}, 
\end{xy}\;\,.
\]
In particular, the map $B_+$ increases the degree of the graph by $1$. 

We define the coproduct $\Delta$ on $\calH$ recursively by multiplicativity and 
\begin{align}\label{Hoch}
\Delta(t)=t \otimes \mathbb{I} + (\mathbb{I} \otimes B_+)\Delta (f)
\end{align}
for $t=B_+(f)$. In words of Hochschild cohomology of bialgebras, the grafting operator $B_+$ satisfies the Hochschild $1$-cocycle condition. For example, we have
\begin{align*}
 \Delta\left(\mathbb{I}\right)
 &=\mathbb{I} \otimes \mathbb{I},\\
 \Delta\left(
 \,\,\begin{xy}
   {(0,0) \ar @{{*}-{*}} (0,0)},
 \end{xy}\,\,
 \right)
 &=\,\begin{xy}
   {(0,0) \ar @{{*}-{*}} (0,0)},
 \end{xy}\,\,
 \otimes \mathbb{I} + \mathbb{I} \otimes 
 \,\begin{xy}
   {(0,0) \ar @{{*}-{*}} (0,0)},
 \end{xy}\,\,, \\
 \Delta\left(
 \,\,\begin{xy}
   {(0,0) \ar @{{*}-{*}} (0,0)}, 
   {(3,0) \ar @{{*}-{*}} (3,0)}, 
 \end{xy}\,\,
 \right)
 &=\,\,\begin{xy}
   {(0,0) \ar @{{*}-{*}} (0,0)}, 
   {(3,0) \ar @{{*}-{*}} (3,0)}, 
 \end{xy}\,\,
 \otimes \mathbb{I} 
 +2\,\,\,\begin{xy}
   {(0,0) \ar @{{*}-{*}} (0,0)},
 \end{xy}\,\,
 \otimes
 \,\begin{xy}
   {(0,0) \ar @{{*}-{*}} (0,0)},
 \end{xy}\,\,
 + \mathbb{I} \otimes 
  \,\,\begin{xy}
   {(0,0) \ar @{{*}-{*}} (0,0)}, 
   {(3,0) \ar @{{*}-{*}} (3,0)}, 
 \end{xy}\,\,, \\
 \Delta\left(
 \,\,\begin{xy}
   {(0,-1.5) \ar @{{*}-{*}} (0,1.5)}, 
 \end{xy}\,\,
 \right)
 &=\,\begin{xy}
   {(0,-1.5) \ar @{{*}-{*}} (0,1.5)}, 
 \end{xy}\,\,
 \otimes \mathbb{I} 
 +\,\begin{xy}
   {(0,0) \ar @{{*}-{*}} (0,0)},
 \end{xy}\,\,
 \otimes
 \,\begin{xy}
   {(0,0) \ar @{{*}-{*}} (0,0)},
 \end{xy}\,\,
 + \mathbb{I} \otimes 
 \,\begin{xy}
   {(0,-1.5) \ar @{{*}-{*}} (0,1.5)}, 
 \end{xy}\,\,, \\
 \Delta\left(
 \,\,\begin{xy}
   {(0,-1.5) \ar @{{*}-{*}} (2,1.5)}, 
   {(2,1.5) \ar @{{*}-{*}} (4,-1.5)},
 \end{xy}\,\,
 \right)
 &= \,\begin{xy}
   {(0,-1.5) \ar @{{*}-{*}} (2,1.5)}, 
   {(2,1.5) \ar @{{*}-{*}} (4,-1.5)},
 \end{xy}\,\,
 \otimes \mathbb{I}
 +\,\,\begin{xy}
   {(0,0) \ar @{{*}-{*}} (0,0)}, 
   {(3,0) \ar @{{*}-{*}} (3,0)}, 
 \end{xy}\,\,
 \otimes
 \,\begin{xy}
   {(0,0) \ar @{{*}-{*}} (0,0)},
 \end{xy}\,\,
 +2 \,\,\,\begin{xy}
   {(0,0) \ar @{{*}-{*}} (0,0)},
 \end{xy}\,\,
 \otimes
 \,\begin{xy}
   {(0,-1.5) \ar @{{*}-{*}} (0,1.5)}, 
 \end{xy}\,\,
 +\mathbb{I} \otimes 
  \,\begin{xy}
   {(0,-1.5) \ar @{{*}-{*}} (2,1.5)}, 
   {(2,1.5) \ar @{{*}-{*}} (4,-1.5)},
 \end{xy}\,\,.
\end{align*}
It is known that the coproduct $\Delta$ is coassociative but not cocommutative. 

The counit $\hat{\I}:\calH\to\Q$ is defined by vanishing on $\Aug(\calH)$ and $\hat{\I}(\I)=1$. If we denote the product by $m:\calH\otimes\calH\to\calH$, we define the antipode $S$ by the anti-automorphism on $\calH$ satisfying
\[
m\circ (S\otimes id)\circ\Delta = \I\circ\hat{\I} = m\circ (id \otimes S)\circ\Delta . 
\]
Then the tuple $(\calH,m,\I,\Delta,\hat{\I},S)$ forms a Hopf algebra known as the Connes-Kreimer Hopf algebra of rooted trees.

\section{Rooted tree maps}\label{rtm}
In this section, we introduce rooted tree maps developed in \cite{TW22}. Let the identity map on $\A_r$ be assigned to the empty tree $\I$, i.e., $\widetilde{\I}=id$. For any rooted forest $f$ of positive degree, we define the $\Q$-linear map $\widetilde{f}:\A_r\to\A_r$ by the following four conditions. 
\begin{itemize}
 \item[(\,I\,)] 
 If $f=\,\begin{xy}{(0,0)\ar@{{*}-{*}}(0,0)}\end{xy}\,\,$, 
 $\widetilde{f}(z^{\delta}_s)=z^{\delta}_s(z-z^{\delta}_s)$ and $\widetilde{f}(z)=0$, 
 \item[(I\hspace{-.01em}I)] 
 $\widetilde{B_+(f)}(z_s^{\delta})
 =R_{z-z^{\delta}_s}R_{2z-z^{\delta}_s}R^{-1}_{z-z^{\delta}_s}\widetilde{f}(z_s^{\delta})$ and $\widetilde{B_+(f)}(z)=0$, 
 \item[(I\hspace{-.15em}I\hspace{-.15em}I)]
 If $f=gh$, $\widetilde{f}(v)=\widetilde{g}(\widetilde{h}(v))$ for $v\in\{z,z_s^{\delta} | s\in\mu_r\}$, 
 \item[(I\hspace{-.15em}V\hspace{-.06em})] 
 $\widetilde{f}(wv)=M(\widetilde{\Delta(f)}(w\otimes v))$ for $w\in\A_r,v\in\{z,z_s^{\delta} | s\in\mu_r\}$, 
\end{itemize}
where $s\in\mu_r$, $R_w$ denotes the right multiplication map by $w$, i.e., $R_w(v)=vw$ for $v,w\in\A_r$, $M:\A_r\otimes\A_r\to\A_r$ denotes the concatenation product, and $\widetilde{\Delta (f)}=\sum_{(f)}\widetilde{a}\otimes \widetilde{b}$ when $\Delta (f)=\sum_{(f)}a\otimes b$. As a matter of fact, the assignment $\widetilde{\ }:\calH\to\End_{\Q}(\A_r)$ is an algebra homomorphism. We find that $\widetilde{f}(z_s^{\delta})$ always ends with $z-z^{\delta}_s$ and hence the condition (I\hspace{-.01em}I) is well-defined. We also find that the image $\widetilde{f}(v)$ in the condition (I\hspace{-.15em}I\hspace{-.15em}I) does not depend on how to decompose $f$ into $g$ and $h$. In fact, the conditions (I\hspace{-.15em}I\hspace{-.15em}I) and (I\hspace{-.15em}V\hspace{-.06em}) hold for any $v\in\A_r$. We call $\widetilde{f}$ the RTM assigned to $f\in\calH$. The commutativity of RTMs is nontrivial but can be proved. 
\begin{ex}[calculations of images of RTMs]
 Since $
  \widetilde{\,\,\begin{xy}
   {(0,0) \ar @{{*}-{*}} (0,0)}, 
  \end{xy}\,\,}
  (z_s^{\delta})
  =z_s^{\delta}(z-z_s^{\delta})
 $
 and
 $
 \Delta\left(
 \,\,\begin{xy}
   {(0,0) \ar @{{*}-{*}} (0,0)},
 \end{xy}\,\,
 \right)
 =\,\begin{xy}
   {(0,0) \ar @{{*}-{*}} (0,0)},
 \end{xy}\,\,
 \otimes \mathbb{I} + \mathbb{I} \otimes 
 \,\begin{xy}
   {(0,0) \ar @{{*}-{*}} (0,0)},
 \end{xy}\,\, 
 $, we have
 \begin{align*}
 \widetilde{\,\,\begin{xy}
   {(0,0) \ar @{{*}-{*}} (0,0)}, 
   {(3,0) \ar @{{*}-{*}} (3,0)}, 
 \end{xy}\,\,}\,(z_s^{\delta})
 =\widetilde{\,\,\begin{xy}
   {(0,0) \ar @{{*}-{*}} (0,0)}, 
 \end{xy}\,\,}\,(z_s^{\delta}(z-z_s^{\delta}))
 =\widetilde{\,\,\begin{xy}
   {(0,0) \ar @{{*}-{*}} (0,0)},
 \end{xy}\,\,}\,(z_s^{\delta})(z-z_s^{\delta})
 + z_s^{\delta} 
 \widetilde{\,\,\begin{xy}
   {(0,0) \ar @{{*}-{*}} (0,0)},
 \end{xy}\,\,}\,(z-z_s^{\delta})
 =z_s^{\delta}(z-z_s^{\delta})^2+(z_s^{\delta})^2(z-z_s^{\delta}).
 \end{align*}
 %
 %
 Then we calculate
 \begin{align*}
  \widetilde{\,\,\begin{xy}
   {(0,-1.6) \ar @{{*}-{*}} (2,1.4)}, 
   {(2,1.4) \ar @{{*}-{*}} (4,-1.6)},
 \end{xy}\,\,}\,(z_s^{\delta})
 &=\widetilde{B_{+}\left(
 \,\,\begin{xy}
   {(0,0) \ar @{{*}-{*}} (0,0)}, 
   {(3,0) \ar @{{*}-{*}} (3,0)}, 
 \end{xy}\,\, 
 \right)}(z_s^{\delta})
 =R_{z-z_s^{\delta}}R_{2z-z_s^{\delta}}R^{-1}_{z-z_s^{\delta}}
  \widetilde{\,\,\begin{xy}
   {(0,0) \ar @{{*}-{*}} (0,0)}, 
   {(3,0) \ar @{{*}-{*}} (3,0)}, 
 \end{xy}\,\,}\,(z_s^{\delta}) \\
 &=z_s^{\delta}(z-z_s^{\delta})(2z-z_s^{\delta})(z-z_s^{\delta})+(z_s^{\delta})^2(2z-z_s^{\delta})(z-z_s^{\delta}).
 \end{align*}
\end{ex}

\section{Harmonic product and Diamond product}\label{har}
The harmonic product $\ast:\A_r\times\A_r\to\A_r$ is defined by $\Q$-bilinearity and
\begin{itemize}
\item[(\,I\,)] $1\ast w=w\ast 1=w$,
\item[(I\hspace{-.01em}I)] $vy_s\ast wy_t=(v\ast wy_t)y_s+(vy_s\ast w)y_t+(v\ast w)xy_{st}$, 
\item[(I\hspace{-.15em}I\hspace{-.15em}I)] $vx\ast w=v\ast wx=(v\ast w)x$,
\end{itemize}
for $v,w\in\A_r$, $s,t\in\mu_r$. It is associative and commutative. The tuples $(\A_r^1,\ast)$ and $(\A_r^0,\ast)$ are subalgebras of $(\A_r,\ast)$. Note that the composition $\mapL\,\mapI$ is known as the evaluation map of MLVs of harmonic type and hence an algebra homomorphism with respect to $\ast$ (see \cite{AK04}). 
\begin{lem}
For $k,l\geq 1, s,t\in\mu_r$, and $v,w\in\A_r$, we have
\begin{itemize}
\item[(i)] $v z_{k,s}\ast w z_{l,t} = (v\ast w z_{l,t})z_{k,s}+(v z_{k,s} \ast w)z_{l,t}+(v\ast w)z_{k+l,st},$
\item[(ii)] $z_{k,s} v\ast z_{l,t} w = z_{k,s} (v\ast z_{l,t} w)+z_{l,t} (z_{k,s} v\ast w)+z_{k+l,st} (v\ast w).$
\end{itemize}
\end{lem}
\begin{proof}
Because of the condition (I\hspace{-.15em}I\hspace{-.15em}I), it is enough to show when $v,w\in\A_r^1$. 

To show (i), substitute $vx^{k-1}$ and $wx^{l-1}$ into $v$ and $w$, respectively, in the condition (I\hspace{-.01em}I) and then use the condition (I\hspace{-.15em}I\hspace{-.15em}I). 

We show (ii) by induction on total degree of words. If $v=w=1$, it follows from (i) for $v=w=1$. 

If $v=v'z_{m,a}\ (v'\in\A_r^1,m\geq 1,a\in\mu_r)$ and $w=1$, the left-hand side equals
\begin{align}\label{eq6}
(z_{k,s}v'\ast z_{l,t})z_{m,a}+z_{k,s}v z_{l,t}+z_{k,s}v'z_{l+m,ta}
\end{align}
because of (i). The 1st term turns into
\[
\{z_{k,s}(v'\ast z_{l,t})+z_{l,t}z_{k,s}v'+z_{k+l,st}v'\}z_{m,a}
\]
by induction, and hence we have
\[
\eqref{eq6} = z_{k,s}\{(v'\ast z_{l,t})z_{m,a}+vz_{l,t}+v'z_{l+m,ta}\}+z_{l,t}z_{k,s}v+z_{k+l,st}v.
\]
Again by (i), we see that this coincides with the right-hand side. The proof goes similarly if $v=1$ and $w=w'z_{n,b}\ (w'\in\A_r,n\geq 1,b\in\mu_r)$. 

If $v=v'z_{m,a}$ and $w=w'z_{n,b}$, the left-hand side equals
\[
(z_{k,s}v'\ast z_{l,t}w)z_{m,a}+(z_{k,s}v\ast z_{l,t}w')z_{n,b}+(z_{k,s}v'\ast z_{l,t}w')z_{m+n,ab}
\]
because of (i). This turns into
\begin{align*}
& \{z_{k,s}(v'\ast z_{l,t}w)+z_{l,t}(z_{k,s}v'\ast w)+z_{k+l,st}(v'\ast w)\}z_{m,a} \\
&+ \{z_{k,s}(v\ast z_{l,t}w')+z_{l,t}(z_{k,s}v\ast w')+z_{k+l,st}(v\ast w')\}z_{n,b} \\
&+ \{z_{k,s}(v'\ast z_{l,t}w')+z_{l,t}(z_{k,s}v'\ast w')+z_{k+l,st}(v'\ast w')\}z_{m+n,ab}
\end{align*}
by induction. Again by (i), we see that this coincides with the right-hand side. 
\end{proof}
From now on, let $y=y_{1}$ for simplicity. For $s\in\mu_{r}$, we define the $\Q$-bilinear map $\diamond_s:\A_1\times\A_r\to\A_r$ by 
\begin{align} \label{diamond}
 \begin{split}
  1\diamond_{s}w & =w,\\
  v\diamond_{s}1 & =\psi_s\varphi(v),\\
  vx\diamond_{s}wx & =(v\diamond_{s}wx)x-(vy\diamond_{s}w)x,\\
  vy\diamond_{s}wx & =(v\diamond_{s}wx)y+(vy\diamond_{s}w)x,\\
  vx\diamond_{s}wy & =(v\diamond_{s}wy)x+(vx\diamond_{s}w)y,\\
  vy\diamond_{s}wy & =(v\diamond_{s}wy)y-(vx\diamond_{s}w)y,\\
  vx\diamond_{s}wy_{t} & =(v\diamond_{s}wy_{t})x+(v\diamond_{s}wz_{t})y_{t}-(vy\diamond_{s}w)y_{t},\\
  vy\diamond_{s}wy_{t} 
  &=(v\diamond_{s}wy_{t})y-(v\diamond_{s}wz_{t})y_{t}+(vy\diamond_{s}w)y_{t},
 \end{split} 
\end{align}
for $v\in\mathcal{A}_{1},w\in\mathcal{A}_{r}$ and $1\neq t\in\mu_r$. 
When $r=1$, 
the product $\diamond_{1}$ corresponds to the one defined in \cite{HMO19} and is commutative. Note that, in general, $1$ is the left unit but not the right unit. For example, one checks $y\diamond_{s}1=z-z_{s}^{\delta}$. 
\begin{lem} \label{lem:x+y}
For $s\in\mu_{r}$, $v\in\A_{1}$, and $w\in\A_{r}$, we have
\[
v z\diamond_{s}w=v \diamond_{s}wz=(v \diamond_{s}w)z.
\]
\end{lem}
\begin{proof}
 By definition, we easily see $v z\diamond_{s}w=(v \diamond_{s}w)z$.
 
 We prove $v \diamond_{s}wz=(v \diamond_{s}w)z$ for words $v,w$ by induction on $d=\deg(v)$.
 It is obvious if $d=0$. Assume $d\ge 1$.
 If $v=v'x$, by definition (in particular, adding 3rd and 5th identities in \eqref{diamond}), we have
 \begin{align*}
  v'x\diamond_{s}wz
  &=(v'\diamond_{s}wz)x-(v'y\diamond_{s}w)x+(v'x\diamond_{s}w)y \\
  &=(v'\diamond_{s}wz)x-(v'z\diamond_{s}w)x+(v'x\diamond_{s}w)z.
 \end{align*}
 By the induction hypothesis, the first two terms cancel out, and hence we obtain the assertion. The proof goes similarly when $v=v'y$. 
\end{proof}
\begin{lem}\label{lem6}
For $s,t\in\mu_{r}$, $v\in\A_1$, and $w\in\A_{r}$, we have
 \begin{itemize}
 \item[(i)] $v x\diamond_{s}w z_{t}^{\delta}=(v \diamond_{s}w z_{t}^{\delta})z_{t}^{\delta}-(v y\diamond_{s}w)z_{t}^{\delta}$, 
 \item[(ii)] $v y\diamond_{s}w z_{t}^{\delta}=(v \diamond_{s}w z_{t}^{\delta})(y \diamond_{t} 1)+(v y\diamond_{s}w)z_{t}^{\delta}$.
 \end{itemize}
\end{lem}
\begin{proof}
 By the 3rd and the 7th identities in \eqref{diamond}, we have (i).  
 By (i), Lemma \ref{lem:x+y}, and $y\diamond_t 1=z-z_t^{\delta}$, we have (ii). 
\end{proof}
We put $z_{s}=x+y_{s}$ for simplicity (, and hence $z_1=z$). Note that $\varphi(z_s)=z_s^{\delta}$. 
\begin{prop}\label{prop0}
For $s\in\mu_r$, $v\in\A_1$, and $w\in\A_r$, we have
\[
v\diamond_{s}w=\psi_{s}(\varphi(v)\ast\psi_{s}^{-1}(w)).
\]
\end{prop}
\begin{proof}
If $v=1$ or $w=1$, it is obvious. Otherwise, the proof goes by induction on $\deg(v)+\deg(w)$. 

If $v=v'z$, by definitions, the right-hand side turns into
\[
\psi_s(\varphi(v'z)\ast\psi_s^{-1}(w)) 
 = \psi_s(\varphi(v')x\ast\psi_s^{-1}(w)) 
 = \psi_s\left((\varphi(v')\ast\psi_s^{-1}(w))x\right) 
 = \psi_s(\varphi(v')\ast\psi_s^{-1}(w))z. 
\]
Then, by the induction hypothesis, this equals $(v'\diamond_s w)z$, which equals the left-hand side because of Lemma \ref{lem:x+y}. 

Similarly, if $w=w'z$, the right-hand side turns into
\[
\psi_s(\varphi(v)\ast\psi_s^{-1}(w'z)) 
 = \psi_s(\varphi(v)\ast\psi_s^{-1}(w')x) 
 = \psi_s\left((\varphi(v)\ast\psi_s^{-1}(w'))x\right) 
 = \psi_s(\varphi(v)\ast\psi_s^{-1}(w'))z, 
\]
which equals the left-hand side. 

To complete the proof, we show when $v=v'x$ and $w=w'z_t^{\delta}$. In this case, the right-hand side turns into
\begin{align}\label{eq8}
\psi_s\left(\varphi(v')z\ast\psi_s^{-1}(w'z_t^{\delta})\right). 
\end{align}
Without loss of generality, suppose $\varphi(w')=z_{k_1,t_1}\cdots z_{k_n,t_n}$. Then, by definitions, we find 
\[
\psi_s^{-1}(w'z_t^{\delta})=\psi_s^{-1}(w')z_{\frac{t}{t_n}}. 
\]
Hence, we have
\begin{align}
 \eqref{eq8} 
 &=\psi_s\left(
  (\varphi(v')y\ast\psi_s^{-1}(w'))z_{\frac{t}{t_n}}
  +(\varphi(v')\ast\psi_s^{-1}(w')y_{\frac{t}{t_n}})z
  +(\varphi(v')\ast\psi_s^{-1}(w'))xz_{\frac{t}{t_n}}
  \right). \label{eq9}
\end{align}
Since $\varphi(v')\in\A_1$, $\psi_s^{-1}(w')=z_{k_1,\frac{t_1}{s}}z_{k_2,\frac{t_2}{t_1}}\cdots z_{k_n,\frac{t_n}{t_{n-1}}}$, and the harmonic product has combinatorial meaning of overlapping shuffle, the subscript of `$y$' in the last $z_{\frac{t}{t_n}}$ or $z$ changes into
\[
s\times\left(\frac{t_1}{s}\times\frac{t_2}{t_1}\times\cdots\times\frac{t_n}{t_{n-1}}\right)\times\frac{t}{t_n}=t \qquad \text{or} \qquad s\times\left(\frac{t_1}{s}\times\frac{t_2}{t_1}\times\cdots\times\frac{t_n}{t_{n-1}}\times\frac{t}{t_n}\right)=t,
\]
respectively, after the map $\psi_s$ applies. Therefore we have
\begin{align*}
\eqref{eq9} &= \psi_s(\varphi(v')y\ast\psi_s^{-1}(w')+\varphi(v')\ast\psi_s^{-1}(w')y_{\frac{t}{t_n}}+(\varphi(v')\ast\psi_s^{-1}(w'))x)z_t^{\delta} \\
 &= \psi_s(-\varphi(v'y)\ast\psi_s^{-1}(w')+\varphi(v')\ast\psi_s^{-1}(w'z_t^{\delta}))z_t^{\delta} \\
 &= (-v'y\diamond_s w'+v'\diamond_s w'z_t^{\delta})z_t^{\delta}. 
\end{align*}
The last equality is by the induction hypothesis. By Lemma \ref{lem6} (i), this coincides with the left-hand side. 
\end{proof}
\begin{lem}\label{lem1}
For $s\in\mu_{r}$, $u,v\in\A_{1}$ and $w\in\A_{r}$, we have
\[
(u\diamond_1 v)\diamond_s w=u\diamond_s (v\diamond_s w). 
\]

%
\end{lem}
\begin{proof}
 We have
 \begin{align*}
  \text{L.H.S.} 
  &= (\varphi(\varphi (u) \ast \varphi (v)))\diamond_s w \\
  &= \psi_s((\varphi (u) \ast \varphi (v)) \ast \psi_s^{-1}(w)) \\
  &= \psi_s (\varphi (u) \ast (\varphi (v) \ast \psi_s^{-1}(w))) \\
  &= \psi_s (\varphi (u) \ast (\psi_s^{-1}(v \diamond_s w)))=\text{R.H.S.} 
 \end{align*}
by Proposition \ref{prop0} and the associativity of $\ast$. 
\end{proof}
\begin{lem} \label{lem:y}
For $s,t\in\mu_{r}$ and $v,w\in\A_r$,
we have
\[
y\diamond_{s}v z_{t}^{\delta}w=(y\diamond_{s}v)z_{t}^{\delta}w + v z_{t}^{\delta}(y\diamond_{s}w) + v z_{t}^{\delta}(z_{s}^{\delta}-z_{t}^{\delta})w.
\]
\end{lem}
\begin{proof}
We prove the lemma by induction on $\deg(w)$. When $w=1$, we
have
\begin{align}\label{eq10}
 y\diamond_{s}vz_{t}^{\delta} = (y\diamond_{s}v)z_t^{\delta}+(1\diamond_{s}vz_t^{\delta})(y\diamond_t 1)
\end{align}
by Lemma \ref{lem6} (ii). Since $y\diamond_t 1=y\diamond_s 1+(z_s^{\delta}-z_t^{\delta})$, we have
\[
\eqref{eq10} = (y\diamond_{s}v)z_{t}^{\delta}+vz_{t}^{\delta}(y\diamond_{s} 1)+vz_{t}^{\delta}(z_{s}^{\delta}-z_{t}^{\delta})
\]
and the assertion. 
If $w=w'z\,(w'\in\mathcal{A}_{r})$,
by the induction hypothesis and Lemma \ref{lem:x+y}, we have 
\begin{align*}
\text{L.H.S.} & =(y\diamond_{s}vz_{t}^{\delta}w')z\\
 & =(y\diamond_{s}v)z_{t}^{\delta}w'z+vz_{t}^{\delta}(y\diamond_{s}w')z+vz_{t}^{\delta}(z_{s}^{\delta}-z_{t}^{\delta})w'z=\text{R.H.S}.
\end{align*}
If $w=w'z_{t'}^{\delta}\,(w'\in\A_{r})$, by
Lemma \ref{lem6} (ii) and the induction hypothesis, we have 
\begin{align*}
\text{L.H.S.} 
 & =(1\diamond_s vz_t^{\delta}w'z_{t'}^{\delta})(y\diamond_{t'} 1)+(y\diamond_s vz_t^{\delta}w')z_{t'}^{\delta} \\ 
 & =(1\diamond_s vz_t^{\delta}w'z_{t'}^{\delta})(y\diamond_{t'} 1)+(y\diamond_{s}v)z_{t}^{\delta}w+vz_{t}^{\delta}(y\diamond_{s}w')z_{t'}^{\delta}+vz_{t}^{\delta}(z_{s}^{\delta}-z_{t}^{\delta})w \\
 & =\text{R.H.S.}
\end{align*}
This finishes the proof. 
\end{proof}
Now write $R=R_{y}R_{x+2y}R_{y}^{-1}$. For rooted forests $f$, we define polynomials $F_{f}\in\mathcal{A}_{1}^{1}$ recursively by 
\begin{itemize}
\item $F_{\mathbb{I}}=1$,
\item $F_{\,\begin{xy}{(0,0)\ar@{{*}-{*}}(0,0)}\end{xy}\,\,}=y$, 
\item $F_{t}=R(F_{f})$ if $t=B_{+}(f)$ and $f\ne\mathbb{I}$, 
\item $F_{f}=F_{g}\diamond_{1}F_{h}$ if $f=gh$. 
\end{itemize}
The subscript of $F$ is extended linearly. 
%
\begin{prop}\label{prop7}
For $f\in\calH$, put $\Delta(f)=\sum_{(f)}f'\otimes f''$. 
Then, for $s,s'\in\mu_{r}$ and $v,w\in\A_{r}$, we have 
\[F_{f}\diamond_{s}v z_{s'}^{\delta}w = \sum_{(f)}(F_{f'}\diamond_{s}v)z_{s'}^{\delta}(F_{f''}\diamond_{s'}w). 
\]
\end{prop}
\begin{proof}
It is enough to consider the case that $f$ is a monomial, i.e., a rooted forest. If $f=\I$, it is obvious. If $f=\,\begin{xy}{(0,0)\ar@{{*}-{*}}(0,0)}\end{xy}\,\,$, by Lemma \ref{lem:y}, we find the proposition holds. 

Assume $\deg(f)\ge2$ and the proposition holds for any elements in $\calH$ of degree less than $\deg(f)$. 
If $f=gh\ (g,h\neq\I)$, we have
\begin{align}\label{eq11}
F_{f}\diamond_{s}vz_{s'}^{\delta}w = (F_{g}\diamond_{1}F_{h})\diamond_{s}vz_{s'}^{\delta}w = F_{g}\diamond_{s}(F_{h}\diamond_{s}vz_{s'}^{\delta}w)
\end{align}
because of Lemma \ref{lem1}. Since $\deg(g),\deg(h)<\deg(f)$, we have
\begin{align}\label{eq12}
\eqref{eq11} = \sum_{(h)}F_{g}\diamond_{s}(F_{h'}\diamond_{s}v)z_{s'}^{\delta}(F_{h''}\diamond_{s'}w) =\sum_{(g)}\sum_{(h)}(F_{g'}\diamond_{s}(F_{h'}\diamond_{s}v))z_{s'}^{\delta}(F_{g''}\diamond_{s'}(F_{h''}\diamond_{s'}w))
\end{align}
by the induction hypothesis. Again by Lemma \ref{lem1}, we have
\[
\eqref{eq12} = \sum_{(g)}\sum_{(h)}((F_{g'}\diamond_{1}F_{h'})\diamond_{s}v)z_{s'}^{\delta}((F_{g''}\diamond_{1}F_{h''})\diamond_{s'}w) = \sum_{(f)}(F_{f'}\diamond_{s}v)z_{s'}^{\delta}(F_{f''}\diamond_{s'}w), 
\]
and hence the assertion. 

If $f$ is a tree and $f=B_{+}(g)$, we have $F_{f}=R(F_{g})$. In this case, the proof goes inductively on $\deg(w)$. 
When $w=1$, we have 
\begin{align}
 F_{f}\diamond_{s}vz_{s'}^{\delta} &=R(F_{g})\diamond_{s}vz_{s'}^{\delta}\nonumber \\
 & =(R_y^{-1}(F_g)x+2F_g)y\diamond_s v z_{s'}^{\delta}\nonumber \\
 & =(F_f \diamond_s v)z_{s'}^{\delta}+((R_y^{-1}(F_g)x+2F_g)\diamond_s v z_{s'}^{\delta})(z-z_{s'}^{\delta}) \label{eq2}
\end{align}
because of Lemma \ref{lem6}. Since $\deg(g)<\deg(f)$, we have
\[
F_{g}\diamond_{s}vz_{s'}^{\delta}=\sum_{(g)}(F_{g'}\diamond_{s}v)z_{s'}^{\delta}(F_{g''}\diamond_{s'}1)
\]
by the induction hypothesis. Then we have 
\begin{align} \eqref{eq2} 
&=(F_f \diamond_s v)z_{s'}^{\delta} + (R_y^{-1}(F_g)x \diamond_s vz_{s'}^{\delta})(z-z_{s'}^{\delta}) + 2\sum_{(g)} (F_{g'} \diamond_s v)z_{s'}^{\delta}(F_{g''} \diamond_{s'} 1)(z-z_{s'}^{\delta}) \nonumber \\
&=(F_f \diamond_s v)z_{s'}^{\delta} + \{(R_y^{-1}(F_g) \diamond_s vz_{s'}^{\delta})z_{s'}^{\delta} - (F_g \diamond_s v)z_{s'}^{\delta}\}(z-z_{s'}^{\delta}) \nonumber \\
&\quad + 2(F_g \diamond_s v)z_{s'}^{\delta}(y \diamond_{s'} 1) + \sum_{\begin{subarray}{c} (g) \\ g''\neq\mathbb{I}\end{subarray}} (F_{g'} \diamond_s v)z_{s'}^{\delta}((R(F_{g''})-R_y^{-1}(F_{g''})xy) \diamond_{s'} 1) \label{eq3}
\end{align}
because of Lemma \ref{lem6} (i), $y\diamond_{s'}1 = z-z_{s'}^{\delta}$, and
\[
(F_{g''} \diamond_{s'} 1)(z-z_{s'}^{\delta})=F_{g''}y\diamond_{s'}1 = \frac{1}{2}(R(F_{g''})-R_y^{-1}(F_{g''})xy) \diamond_{s'} 1. 
\]
By the way, we find
\begin{align*}
\sum_{(f)} (F_{f'} \diamond_s v) z_{s'}^{\delta} (F_{f''} \diamond_{s'} 1) 
&= \sum_{\begin{subarray}{c} (f) \\ f''\neq\I \end{subarray}} (F_{f'} \diamond_s v) z_{s'}^{\delta} (F_{f''} \diamond_{s'} 1) + (F_f \diamond_s v) z_{s'}^{\delta}\\
&= \sum_{(g)} (F_{g'} \diamond_s v) z_{s'}^{\delta} (F_{B_+(g'')} \diamond_{s'} 1) + (F_f \diamond_s v) z_{s'}^{\delta}\\
&= \sum_{\begin{subarray}{c} (g) \\ g''\neq\I \end{subarray}} (F_{g'} \diamond_s v) z_{s'}^{\delta} (R(F_{g''}) \diamond_{s'} 1) + (F_g \diamond_s v) z_{s'}^{\delta} (y \diamond_{s'} 1) + (F_f \diamond_s v) z_{s'}^{\delta}
\end{align*}
because of \eqref{Hoch}. 
Therefore we have
\begin{align}\label{eq4}
 \begin{split}
 \eqref{eq3}
 &=\sum_{(f)} (F_{f'} \diamond_s v)z_{s'}^{\delta}(F_{f''} \diamond_{s'} 1) 
  +(R_y^{-1}(F_g) \diamond_s v z_{s'}^{\delta})z_{s'}^{\delta}(z-z_{s'}^{\delta}) \\
  &\qquad -\sum_{\begin{subarray}{c} (g) \\ g''\neq\I\end{subarray}} 
  (F_{g'} \diamond_s v)z_{s'}^{\delta}(R_y^{-1}(F_{g''})xy \diamond_{s'} 1). 
 \end{split} 
\end{align}
We now see that the 2nd and the 3rd terms in \eqref{eq4} cancel out. To see this, we need to show
\[
\sum_{\begin{subarray}{c} (g) \\ g''\neq\I\end{subarray}} 
  (F_{g'} \diamond_s v)z_{s'}^{\delta}(R_y^{-1}(F_{g''}) \diamond_{s'} 1) = R_y^{-1}(F_g) \diamond_s v z_{s'}^{\delta}
\]
because of $R_y^{-1}(F_{g''})xy \diamond_{s'} 1 = (R_y^{-1}(F_{g''}) \diamond_{s'} 1)z_{s'}^{\delta}(z-z_{s'}^{\delta})$. By $R_y^{-1}(F_{g''}) \diamond_{s'} 1 = R_{z-z_{s'}^{\delta}}^{-1}(F_{g''}\diamond_{s'}1)$, the induction hypothesis, and Lemma \ref{lem6}, we have
\begin{align*}
\sum_{\begin{subarray}{c} (g) \\ g''\neq\mathbb{I}\end{subarray}} (F_{g'} \diamond_s v)z_{s'}^{\delta}(R_y^{-1}(F_{g''}) \diamond_{s'} 1) 
&= R_{z-z_{s'}^{\delta}}^{-1}\left(\sum_{\begin{subarray}{c} (g) \\ g''\neq\mathbb{I}\end{subarray}} (F_{g'} \diamond_s v)z_{s'}^{\delta}(F_{g''} \diamond_{s'} 1)\right) \\
&= R_{z-z_{s'}^{\delta}}^{-1}(F_g \diamond_s v z_{s'}^{\delta} - (F_g \diamond_{s'} v)z_{s'}^{\delta}) \\
&= R_{z-z_{s'}^{\delta}}^{-1}(R_y^{-1}(F_g) \diamond_s v z_{s'}^{\delta})(z-z_{s'}^{\delta}) \\
&=R_y^{-1}(F_g) \diamond_s v z_{s'}^{\delta}.
\end{align*}
Thus we conclude
\[
\eqref{eq4} = \sum_{(f)} (F_{f'} \diamond_s v)z_{s'}^{\delta}(F_{f''} \diamond_{s'} 1). 
\]

Now we proceed to the case when $\deg(w)\ge 1$. If $w=w'z\,(w'\in\mathcal{A}_{r})$,
we have 
\[
F_{f}\diamond_{s}vz_{s'}^{\delta}w = (F_{f}\diamond_{s}vz_{s'}^{\delta}w')z = \sum_{(f)}(F_{f'}\diamond_{s}v)z_{s'}^{\delta}(F_{f''}\diamond_{s'}w')z = \sum_{(f)}(F_{f'}\diamond_{s}v)z_{s'}^{\delta}(F_{f''}\diamond_{s'}w)
\]
by Lemma \ref{lem:x+y} and the induction hypothesis. If $w=w'z_{s''}^{\delta}\,(w'\in\A_r,s''\in\mu_r)$,
since we have already proved the identity when $w=1$,
we have 
\begin{align*}
F_{f}\diamond_{s}vz_{s'}^{\delta}w & =\sum_{(f)}(F_{f'}\diamond_{s}vz_{s'}^{\delta}w')z_{s''}^{\delta}(F_{f''} \diamond_{s''} 1)\\
 & =\sum_{(f)}\sum_{(f')}(F_{f'_{a}}\diamond_{s}v)z_{s'}^{\delta}(F_{f'_{b}}\diamond_{s'}w)z_{s''}^{\delta}(F_{f''} \diamond_{s''} 1),
\end{align*}
where we put $\Delta(f')=\sum_{(f')}f'_{a}\otimes f'_{b}$. We
also have 
\begin{align*}
\sum_{(f)}(F_{f'}\diamond_{s}v)z_{s'}^{\delta}(F_{f''}\diamond_{s'}w) & =\sum_{(f)}(F_{f'}\diamond_{s}v)z_{s'}^{\delta}(F_{f''}\diamond_{s'}w'z_{s''}^{\delta})\\
 & =\sum_{(f)}\sum_{(f'')}(F_{f'}\diamond_{s}v)z_{s'}^{\delta}(F_{f''_{a}}\diamond_{s'}w')z_{s''}^{\delta}(F_{f''_{b}} \diamond_{s''} 1),
\end{align*}
where we put $\Delta(f'')=\sum_{(f'')}f''_{a}\otimes f''_{b}$. By
the coassociativity of $\Delta$, these two coincide and hence we have conclusion. 
\end{proof}


The following property plays an important role in our proof of Theorem \ref{main3} in Section \ref{pf2}. 
\begin{prop} \label{prop:main3}
 For $s,t\in\mu_r$, $v\in\A_1$, and $w\in\A_r$, we have
 \[
 z_s^{\delta}\bigl(v y \diamond_s w(z-z_t^{\delta})\bigr)=-\tau \bigl(\tau(v)y\diamond_t \tau(w)(z-z_s^{\delta})\bigr)(z-z_t^{\delta}). 
 \]
\end{prop}
\begin{proof}
By Lemma \ref{lem6} (ii) and Lemma \ref{lem2} (ii), it is equivalent to show the identity
\begin{equation}\label{eq17}
vy \diamond_s w - v \diamond_s wz_t^{\delta} = \tau\left(\tau(v) \diamond_t \tau(w)z_s^{\delta} - \tau(v)y \diamond_t \tau(w)\right). 
\end{equation}
We prove this by induction on $\deg(v)+\deg(w)$. Let us start with considering the case of $v=1$. If $w=1,z,z-z_u^{\delta}$ ($u\in\mu_r$), we calculate that both sides turn into 
\[
z-z_s^{\delta}-z_t^{\delta},\quad z(z-z_t^{\delta})-z_s^{\delta}z,\quad 
(z-z_s^{\delta}-z_u^{\delta})(z-z_u^{\delta})-(z-z_u^{\delta})z_t^{\delta}, 
\]
respectively. If $w=w'z$, we calculate
\begin{align*}
\tau(\text{R.H.S.}) &= z\tau(w')z_s^{\delta} - (y \diamond_t 1)(1\diamond_t z\tau(w'))-z(y \diamond_t \tau(w'))+z(y \diamond_t 1)(1\diamond_t \tau(w')) \\
&=\tau((y \diamond_s w')z)-\tau(w'zz_t^{\delta}) = \tau(\text{L.H.S.})
\end{align*}
by Lemma \ref{lem2} (iii) and the induction hypothesis. 
If $w=w'(z-z_u^{\delta})$, we have
\[
\text{L.H.S.} = (y \diamond_s w')(z-z_u^{\delta}) - w'z_u^{\delta}(z-z_u^{\delta}) - w'(z-z_u^{\delta})z_t^{\delta}
\]
by Lemma \ref{lem2} (ii), and
\begin{align*}
\tau(\text{R.H.S.}) &= z_u^{\delta}\tau(w')z_s^{\delta} - (y \diamond_t 1)z_u^{\delta}\tau(w') -z_u^{\delta}(y \diamond_u \tau(w')) \\
&=z_u^{\delta}\tau(y \diamond_s w' - 1\diamond_s w'z_u^{\delta}) - (y \diamond_t 1)z_u^{\delta}\tau(w')
\end{align*}
by Lemma \ref{lem2} (iv) and the induction hypothesis. Thus \eqref{eq17} holds. 

The 2nd case: $v=z$. If $w=1,z,z-z_u^{\delta}$, we calculate that both sides turn into 
\[
z^2-zz_s^{\delta}-z_t^{\delta}z,\quad z(z-z_s^{\delta}-z_t^{\delta})z,\quad 
(z^2-zz_s^{\delta}-z_u^{\delta}z)(z-z_u^{\delta})-(z-z_u^{\delta})z_t^{\delta}z,  
\]
respectively. If $w=w'z$, we calculate
\begin{align*}
 \tau(\text{R.H.S.}) 
 &=z(z \diamond_t \tau(w')z_s^{\delta}) 
  -z\left(y \diamond_t \tau(w) + zy \diamond_t \tau(w') 
  -z(y \diamond_t \tau(w'))\right) \\
 &=z\tau(zy \diamond_s w' - z \diamond_s w'z_t^{\delta}) 
  -z\left(y \diamond_t \tau(w) 
  -z(y \diamond_t \tau(w'))\right)
\end{align*}
by Lemma \ref{lem2} (v) and the induction hypothesis. Applying Lemma \ref{lem2} (iii) to the 3rd term, we find that this is $\tau(\text{L.H.S.})$. 
Note that
\begin{align} \label{eq:tuika2}
 xv\diamond_s z_t^\delta w
 =z_s^\delta (v\diamond_s z_t^\delta w) 
  +z_t^\delta (xv\diamond_t w)
  -zz_t^\delta (v\diamond_t w)
\end{align}
by Lemma \ref{lem2} (iv) and (vi).
If $w=w'(z-z_u^{\delta})$, we have
\begin{align*}
\tau(\text{R.H.S.}) &= z_u^{\delta}(z \diamond_u \tau(w')z_s^{\delta}) - z(y \diamond_t z_u^{\delta}\tau(w')) - z_u^{\delta}(zy \diamond_u \tau(w')) + zz_u^{\delta}(y \diamond_u \tau(w')) \\
&=z_u^\delta\tau(zy \diamond_s w' - z \diamond_s w'z_u^{\delta}) - z(y \diamond_t 1)(1 \diamond_t z_u^{\delta}\tau(w'))
\end{align*}
by \eqref{eq:tuika2} and the induction hypothesis. This is $\tau(\text{L.H.S.})$ because of Lemma \ref{lem2} (ii). Thus \eqref{eq17} holds. 

The 3rd case: $v=y$. If $w=1,z,z-z_u^{\delta}$, we calculate that both sides turn into 
\begin{align*}
& (z-z_s^{\delta}-z_t^{\delta})(z-z_t^{\delta})-(z-z_s^{\delta})z_s^{\delta},\quad (z-z_s^{\delta})^2z-(z-z_s^{\delta})zz_t^{\delta}-zz_t^{\delta}(z-z_t^{\delta}), \\ &(z-z_s^{\delta}-z_u^{\delta})(z-z_u^{\delta})(z-z_t^{\delta}-z_u^{\delta})-(z-z_s^{\delta})z_s^{\delta}(z-z_u^{\delta})-(z-z_u^{\delta})z_t^{\delta}(z-z_t^{\delta}), 
\end{align*}
respectively. 
Note that
\begin{align} \label{eq:tuika}
 xv\diamond_s zw
 =z_s^\delta (v\diamond_s zw) 
  +z(xv\diamond_s w)
  -zz_s^\delta (v\diamond_s w)
\end{align}
by Lemma \ref{lem2} (iii) and (v).
If $w=w'z$, we calculate
\begin{align*}
 \tau(\text{R.H.S.}) 
 &=z_t^{\delta}(1 \diamond_t z\tau(w')z_s^{\delta})+z(x \diamond_t \tau(w')z_s^{\delta})-zz_t^{\delta}(1 \diamond_t \tau(w')z_s^{\delta}) \\
&\qquad -\left(z_t^{\delta}(y \diamond_t z\tau(w')+z(xy \diamond_t \tau(w'))-zz_t^{\delta}(y \diamond_t \tau(w')\right) \\
&=z_t^{\delta}\tau(y \diamond_s w - 1 \diamond_s wz_t^{\delta})+z\tau(y^2 \diamond_s w' - y \diamond_s w'z_t^{\delta})-zz_t^{\delta}(y \diamond_s w' - 1 \diamond_s w'z_t^{\delta})
\end{align*}
by 
\eqref{eq:tuika} and the induction hypothesis. Hence, applying $\tau$ and using Lemma \ref{lem6} (ii), we find \eqref{eq17} also holds in this case. If $w=w'(z-z_u^{\delta})$, we have
\begin{align*}
 \tau(\text{R.H.S.}) 
 &=z_t^{\delta}(1 \diamond_t \tau(w)z_s^{\delta}) 
  +z_u^{\delta}(x \diamond_u \tau(w')z_s^{\delta}) 
  -zz_u^{\delta}(1 \diamond_u \tau(w')z_s^{\delta}) \\
 &\qquad - \left(z_t^{\delta}(y \diamond_t \tau(w)) 
  +z_u^{\delta}(xy \diamond_u \tau(w')) 
  -zz_u^{\delta}(y \diamond_u \tau(w'))\right) \\
 &=z_t^{\delta}\tau(y \diamond_s w -1 \diamond_s wz_t^{\delta}) 
  +z_u^{\delta}\tau(y^2 \diamond_s w' -y \diamond_s w'z_u^{\delta}) 
  -zz_u^{\delta}\tau(y \diamond_s w' - 1 \diamond_s w'z_u^{\delta})
\end{align*}
by \eqref{eq:tuika2} and the induction hypothesis. This is $\tau(\text{L.H.S.})$ because of Lemmas \ref{lem2} (ii) and \ref{lem6} (ii). Thus \eqref{eq17} holds. 

The 4th case: $v=v'z$. If $w=1$, 
\begin{align*}
\tau(\text{R.H.S.}) &= z(\tau(v') \diamond_t z_s^{\delta})+(z_s^{\delta}z-zz_s^{\delta})(\tau(v') \diamond_s 1)-z(\tau(v')y \diamond_t 1) \\
&=z_s^{\delta}z\tau(v' \diamond_s 1)-\tau(v \diamond_s z_t^{\delta})
\end{align*}
by Lemma \ref{lem2} (vi), (i), the induction hypothesis, and Lemma \ref{lem2} (vii). This is $\tau(\text{L.H.S.})$ because of Lemmas \ref{lem:x+y} and \ref{lem2} (i). If $w=z$, 
\begin{align*}
\tau(\text{R.H.S.}) &= z\left(\tau(v') \diamond_t zz_s^{\delta} + \tau(v) \diamond_t z_s^{\delta} -z(\tau(v') \diamond_t z_s^{\delta})\right) -z\left(\tau(v')y \diamond_t z + \tau(v)y \diamond_t 1 -z(\tau(v')y \diamond_t 1)\right) \\
&=z\tau(v'y \diamond_s z - v' \diamond_s zz_t^{\delta}) + z\tau(vy \diamond_s 1 - v \diamond_s z_t^{\delta}) - z\tau(v'y \diamond_s 1 - v' \diamond_s z_t^{\delta})
\end{align*}
by Lemma \ref{lem2} (v) and the induction hypothesis. This is $\tau(\text{L.H.S.})$ because of Lemma \ref{lem:x+y}. If $w=z-z_u^{\delta}$, 
\begin{align*}
\tau(\text{R.H.S.}) &= z(\tau(v') \diamond_t z_u^{\delta}z_s^{\delta}) + z_u^{\delta}(\tau(v) \diamond_u z_s^{\delta}) - zz_u^{\delta}(\tau(v') \diamond_u z_s^{\delta}) \\
&\qquad -\left(z(\tau(v')y \diamond_t z_u^{\delta}) + z_u^{\delta}(\tau(v)y \diamond_u 1) - zz_u^{\delta}(\tau(v')y \diamond_u 1)\right) \\
&=z\tau(v'y \diamond_s (z-z_u^{\delta}) - v' \diamond_s (z-z_u^{\delta})z_t^{\delta}) + z_u^{\delta}\tau(vy \diamond_s 1 - v \diamond_s z_u^{\delta}) - zz_u^{\delta}\tau(v'y \diamond_s 1 - v' \diamond_s z_u^{\delta})
\end{align*}
by Lemma \ref{lem2} (vi) and the induction hypothesis. This is $\tau(\text{L.H.S.})$ because of Lemmas \ref{lem:x+y} and \ref{lem6} (ii). If $w=w'z$, 
\begin{align*}
\tau(\text{R.H.S.}) &= z\left(\tau(v') \diamond_t \tau(w)z_s^{\delta} + \tau(v) \diamond_t \tau(w')z_s^{\delta} - z(\tau(v') \diamond_t \tau(w')z_s^{\delta})\right) \\
&\qquad -z\left(\tau(v')y \diamond_t \tau(w) + \tau(v)y \diamond_t \tau(w') - z(\tau(v')y \diamond_t \tau(w'))\right) \\
&=z\tau(v'y \diamond_s w - v' \diamond_s wz_t^{\delta}) + z\tau(vy \diamond_s w' - v \diamond_s w'z_t^{\delta}) - z\tau(v'y \diamond_s w' - v' \diamond_s w'z_t^{\delta})
\end{align*}
by Lemma \ref{lem2} (v) and the induction hypothesis. This is $\tau(\text{L.H.S.})$ because of Lemma \ref{lem:x+y}. If $w=w'(z-z_u^{\delta})$, 
\begin{align*}
\tau(\text{R.H.S.}) &= z(\tau(v') \diamond_t \tau(w)z_s^{\delta}) + z_u^{\delta}(\tau(v) \diamond_u \tau(w')z_s^{\delta}) - zz_u^{\delta}(\tau(v') \diamond_u \tau(w')z_s^{\delta}) \\
&\qquad -\left(z(\tau(v')y \diamond_t \tau(w)) + z_u^{\delta}(\tau(v)y \diamond_u \tau(w')) - zz_u^{\delta}(\tau(v')y \diamond_u \tau(w'))\right) \\
&=z\tau(v'y \diamond_s w - v' \diamond_s wz_t^{\delta}) + z_u^{\delta}\tau(vy \diamond_s w' - v \diamond_s w'z_u^{\delta}) - zz_u^{\delta}\tau(v'y \diamond_s w' - v' \diamond_s w'z_u^{\delta})
\end{align*}
by Lemma \ref{lem2} (vi) and the induction hypothesis. This is $\tau(\text{L.H.S.})$ because of Lemmas \ref{lem:x+y} and \ref{lem6} (ii). Thus \eqref{eq17} holds. 

The final case: $v=v'y$. If $w=1$, 
\begin{align*}
\tau(\text{R.H.S.}) &= z_t^{\delta}(\tau(v') \diamond_t z_s^{\delta}) + z_s^{\delta}(\tau(v) \diamond_s 1) -zz_s^{\delta}(\tau(v') \diamond_s 1) - z_t^{\delta}(\tau(v')y \diamond_t 1) \\
&=\tau(v'y \diamond_s 1 - v' \diamond_s z_t^{\delta}) + z_s^{\delta}(\tau(v) \diamond_s 1) -zz_s^{\delta}(\tau(v') \diamond_s 1)
\end{align*}
by \eqref{eq:tuika2}, Lemma \ref{lem2} (i), $x \diamond_t 1=z_t^{\delta}$, and the induction hypothesis. This is $\tau(\text{L.H.S.})$ because of Lemmas \ref{lem2} (i) and \ref{lem6} (ii). If $w=z$, 
\begin{align*}
\tau(\text{R.H.S.}) &= z_t^{\delta}(\tau(v') \diamond_t zz_s^{\delta}) + z(\tau(v) \diamond_t z_s^{\delta}) - zz_t^{\delta}(\tau(v') \diamond_t z_s^{\delta}) \\
&\qquad -z_t^{\delta}(\tau(v')y \diamond_t z) + z(\tau(v)y \diamond_t 1) - zz_t^{\delta}(\tau(v')y \diamond_t 1) \\
&=z_t^{\delta}\tau(v'y \diamond_s z - v' \diamond_s zz_t^{\delta}) + z\tau(vy \diamond_s 1 - v \diamond_s z_t^{\delta}) - zz_t^{\delta}\tau(v'y \diamond_s 1 - v' \diamond_s z_t^{\delta})
\end{align*}
by \eqref{eq:tuika} 
and the induction hypothesis. This is $\tau(\text{L.H.S.})$ because of Lemmas \ref{lem:x+y} and \ref{lem6} (ii). If $w=z-z_u^{\delta}$, 
\begin{align*}
\tau(\text{R.H.S.}) &= z_t^{\delta}(\tau(v') \diamond_t z_u^{\delta}z_s^{\delta}) + z_u^{\delta}(\tau(v) \diamond_u z_s^{\delta}) - zz_u^{\delta}(\tau(v') \diamond_u z_s^{\delta}) \\
&\qquad -\left(z_t^{\delta}(\tau(v')y \diamond_t z_u^{\delta}) + z_u^{\delta}(\tau(v)y \diamond_u 1) - zz_u^{\delta}(\tau(v')y \diamond_u 1)\right) \\
&=z_t^{\delta}\tau(v'y \diamond_s (z-z_u^{\delta}) - v' \diamond_s (z-z_u^{\delta})z_t^{\delta}) + z_u^{\delta}\tau(vy \diamond_s 1 - v \diamond_s z_u^{\delta}) - zz_u^{\delta}\tau(v'y \diamond_s 1 - v' \diamond_s z_u^{\delta})
\end{align*}
by \eqref{eq:tuika2} and the induction hypothesis. This is $\tau(\text{L.H.S.})$ because of Lemmas \ref{lem:x+y} and \ref{lem6} (ii). If $w=w'z$, 
\begin{align*}
\tau(\text{R.H.S.}) &= z_t^{\delta}(\tau(v') \diamond_t \tau(w)z_s^{\delta}) + z(\tau(v) \diamond_t \tau(w')z_s^{\delta}) - zz_t^{\delta}(\tau(v') \diamond_t \tau(w')z_s^{\delta}) \\
&\qquad -\left(z_t^{\delta}(\tau(v')y \diamond_t \tau(w)) + z(\tau(v)y \diamond_t \tau(w')) - zz_t^{\delta}(\tau(v')y \diamond_t \tau(w'))\right) \\
&=z_t^{\delta}\tau(v \diamond_s w - v' \diamond_s wz_t^{\delta}) + z\tau(vy \diamond_s w' - v \diamond_s w'z_t^{\delta}) - zz_t^{\delta}\tau(v \diamond_s w' - v' \diamond_s w'z_t^{\delta})
\end{align*}
by \eqref{eq:tuika} 
and the induction hypothesis. This is $\tau(\text{L.H.S.})$ because of Lemmas \ref{lem:x+y} and \ref{lem6} (ii). If $w=w'(z-z_u^{\delta})$, 
\begin{align*}
\tau(\text{R.H.S.}) &= z_t^{\delta}(\tau(v') \diamond_t \tau(w)z_s^{\delta}) + z_u^{\delta}(\tau(v) \diamond_u \tau(w')z_s^{\delta}) - zz_u^{\delta}(\tau(v') \diamond_u \tau(w')z_s^{\delta}) \\
&\qquad -\left(z_t^{\delta}(\tau(v')y \diamond_t \tau(w)) + z_u^{\delta}(\tau(v)y \diamond_u \tau(w')) - zz_u^{\delta}(\tau(v')y \diamond_u \tau(w'))\right) \\
&=z_t^{\delta}\tau(v'y \diamond_s w - v' \diamond_s wz_t^{\delta}) + z_u^{\delta}\tau(vy \diamond_s w' - v \diamond_s w'z_u^{\delta}) - zz_u^{\delta}\tau(v'y \diamond_s w' - v' \diamond_s w'z_u^{\delta})
\end{align*}
by \eqref{eq:tuika2} and the induction hypothesis. This is $\tau(\text{L.H.S.})$ because of Lemmas \ref{lem:x+y} and \ref{lem6} (ii). Thus \eqref{eq17} holds and we complete the proof. 
\end{proof}
\begin{lem}\label{lem2}
For $s,t\in\mu_r, v,v'\in\A_1$, and $w\in\A_r$, we have the followings. 
\begin{itemize}
\item[(i)] $vv' \diamond_s 1 = (v \diamond_s 1)(v' \diamond_s 1)$. 
\item[(ii)] $vy \diamond_s w(z-z_t^{\delta}) = (vy \diamond_s w - v \diamond_s wz_t^{\delta})(y \diamond_t 1)$.
\item[(iii)] $yv \diamond_s zw = (y \diamond_s 1)(v \diamond_s zw)+z(yv \diamond_s w)-z(y \diamond_s 1)(v \diamond_s w)$. 
\item[(iv)] $yv \diamond_s z_t^{\delta}w = (y \diamond_s 1)(v \diamond_s z_t^{\delta}w)+z_t^{\delta}(yv \diamond_t w)$. 
\item[(v)] $zv \diamond_s zw = z(v \diamond_s zw+zv \diamond_s w-z(v \diamond_s w))$. 
\item[(vi)] $zv \diamond_s z_t^{\delta}w = z(v \diamond_s z_t^{\delta}w)+z_t^{\delta}(zv \diamond_t w)-zz_t^{\delta}(v \diamond_t w)$. 
\item[(vii)] $\tau (v \diamond_s 1)=\tau (v) \diamond_s 1$. 
\end{itemize}
\end{lem}
\begin{proof}
\underline{(i):} If $v=1$ or $v'=1$, it is obvious. Otherwise, it is enough to show when $v=z^{k_1-1}y\cdots z^{k_m-1}y$ and $v'=z^{l_1-1}y\cdots z^{l_n-1}y$. One calculates
\[
vv'\diamond_s 1=z^{k_1-1}(z-z_s^{\delta})\cdots z^{k_m-1}(z-z_s^{\delta})z^{l_1-1}(z-z_s^{\delta})\cdots z^{l_n-1}(z-z_s^{\delta}), 
\]
which is clearly equal to $(v \diamond_s 1)(v' \diamond_s 1)$. 

\underline{(ii):} This is a direct consequence of Lemmas \ref{lem:x+y} and \ref{lem6} (ii). 


\underline{(iii):} We first consider the case $v=1$. If $w=1$, it is obvious because of Lemma \ref{lem:x+y}. If $w=w'z\,(w'\in\A_r)$, the left-hand side turns into 
\[
 (y \diamond_s zw')z 
 =((y \diamond_s 1)zw'+z(y \diamond_s w')-z(y \diamond_s 1)w')z
\]
by Lemma \ref{lem:x+y} and the induction hypothesis on degree of words. This is equal to the right-hand side again by Lemma \ref{lem:x+y}. If $w=w'z_t^{\delta}\,(w'\in\A_r,t\in\mu_r)$, the left-hand side turns into 
\[
(y \diamond_s zw')z_t^{\delta}+zw(y \diamond_t 1) = (y \diamond_s 1)zw+z(y \diamond_s w')z_t^{\delta}-z(y \diamond_s 1)w+zw(y \diamond_t 1)
\]
by Lemma \ref{lem6} (ii) and the induction hypothesis. This is equal to the right-hand side again by Lemma \ref{lem6} (ii). If $v=z$, by Lemma \ref{lem:x+y}, we have
\[
\text{L.H.S.} = (y \diamond_s zw)z
\]
and
\[
\text{R.H.S.} = \left((y \diamond_s 1)(1 \diamond_s zw)+z(y \diamond_s w)-z(y \diamond_s 1)(1 \diamond_s w)\right)z, 
\]
which are equal as shown just before. If $v=y$, we need to show when $w=1, w'z, w'z_t^{\delta}$ ($w'\in\A_r, t\in\mu_r$). If $w=1$, 
\[
\text{L.H.S.} = (y^2 \diamond_s 1)z = (y \diamond_s 1)(y \diamond_s z)
\]
and
\[
\text{R.H.S.} = (y \diamond_s 1)(y \diamond_s z)+z(y^2 \diamond_s 1)-z(y \diamond_s 1)^2, 
\]
which coincide. If $w=w'z$, by induction on degree of words, the left-hand side turns into
\[
(y^2 \diamond_s zw')z = (y \diamond_s 1)(y \diamond_s zw')z+z(y^2 \diamond_s w')z-z(y \diamond_s 1)(y \diamond_s w')z,
\]
which is equal to the right-hand side due to Lemma \ref{lem:x+y}. If $w=w'z_t^{\delta}$, by using Lemma \ref{lem6} (ii) and the induction hypothesis, one calculates
\begin{align*}
\text{L.H.S.} &= (y \diamond_s zw)(y \diamond_t 1)+(y^2 \diamond_s zw')z_t^{\delta} \\
 &= \left((y \diamond_s 1)zw+z(y \diamond_s w)-z(y \diamond_s 1)w\right)(y \diamond_t 1) \\
 &\qquad +\left((y \diamond_s 1)(y \diamond_s zw')+z(y^2 \diamond_s w')-z(y \diamond_s 1)(y \diamond_s w')\right)z_t^{\delta}
\end{align*}
and
\begin{align*}
\text{R.H.S.} &= (y \diamond_s 1)\left(zw(y \diamond_t 1)+(y \diamond_s zw')z_t^{\delta}\right)+z\left((y \diamond_s w)(y \diamond_t 1)+(y^2 \diamond_s w')z_t^{\delta}\right) \\
 &\qquad -z(y \diamond_s 1)\left(w(y \diamond_t 1)+(y \diamond_s w')z_t^{\delta}\right), 
\end{align*}
which are equal. If $v=v'z$, it is obvious by Lemma \ref{lem:x+y} and the induction hypothesis. If $v=v'y$, we need to show when $w=1, w'z, w'z_t^{\delta}$ ($w'\in\A_r, t\in\mu_r$). If $w=1$, 
\[
\text{L.H.S.} = yv \diamond_s z = (yv \diamond_s 1)z
\]
and
\[
\text{R.H.S.} = (y \diamond_s 1)(v \diamond_s z)+z(yv \diamond_s 1)-z(y \diamond_s 1)(v \diamond_s 1), 
\]
which are equal. If $w=w'z$, it is obvious by Lemma \ref{lem:x+y} and the induction hypothesis. If $w=w'z_t^{\delta}$, by using Lemma \ref{lem6} (ii) and the induction hypothesis, one calculates
\begin{align*}
\text{L.H.S.} &= (yv' \diamond_s zw)(y \diamond_t 1)+(yv \diamond_s zw')z_t^{\delta} \\
&= \left((y \diamond_s 1)(v' \diamond_s zw)+z(yv' \diamond_s w)-z(y \diamond_s 1)(v' \diamond_s w)\right)(y \diamond_t 1) \\
&\qquad +\left((y \diamond_s 1)(v \diamond_s zw')+z(yv \diamond_s w')-z(y \diamond_s 1)(v \diamond_s w')\right)z_t^{\delta}
\end{align*}
and
\begin{align*}
\text{R.H.S.} &= (y \diamond_s 1)\left((v' \diamond_s zw)(y \diamond_t 1)+(v \diamond_s zw')z_t^{\delta}\right)+z\left((yv' \diamond_s w)(y \diamond_t 1)+(yv \diamond_s w')z_t^{\delta}\right) \\
&\qquad -z(y \diamond_s 1)\left((v' \diamond_s w)(y \diamond_t 1)+(v \diamond_s w')z_t^{\delta}\right), 
\end{align*}
which coincide. 

\underline{(iv):} We first consider the case $v=1$. If $w=1$, it is obvious because of Lemma \ref{lem6} (ii). If $w=w'z\,(w'\in\A_r)$, the left-hand side turns into 
\[
 (y \diamond_s z_t^{\delta}w')z
 =((y \diamond_s 1)z_t^{\delta}w'+z_t^{\delta}(y \diamond_t w'))z
\]
by Lemma \ref{lem:x+y} and the induction hypothesis on degree of words. This is equal to the right-hand side again by Lemma \ref{lem:x+y}. If $w=w'z_u^{\delta}\,(w'\in\A_r,u\in\mu_r)$, the left-hand side turns into 
\[
(y \diamond_s z_t^{\delta}w')z_u^{\delta}+z_t^{\delta}w(y \diamond_u 1) = (y \diamond_s 1)z_t^{\delta}w+z_t^{\delta}(y \diamond_t w')z_u^{\delta}+z_t^{\delta}w(y \diamond_u 1)
\]
by Lemma \ref{lem6} (ii) and the induction hypothesis. This is equal to the right-hand side again by Lemma \ref{lem6} (ii). If $v=z$, by Lemma \ref{lem:x+y}, we have
\[
\text{L.H.S.} = (y \diamond_s z_t^{\delta}w)z
\]
and
\[
\text{R.H.S.} = \left((y \diamond_s 1)(1 \diamond_s z_t^{\delta}w)+z_t^{\delta}(y \diamond_t w)\right)z, 
\]
which are equal as shown just before. If $v=y$, we need to show when $w=1, w'z, w'z_u^{\delta}$ ($w'\in\A_r, u\in\mu_r$). If $w=1$, 
\[
 \text{L.H.S.} 
 =(y \diamond_s z_t^{\delta})(y \diamond_t 1)
  +(y^2 \diamond_s 1)z_t^{\delta}
\]
and
\[
\text{R.H.S.} = (y \diamond_s 1)(y \diamond_s z_t^{\delta})+z_t^{\delta}(y^2 \diamond_t 1), 
\]
which are equal because of Lemma \ref{lem6} (ii). If $w=w'z$, it is obvious by Lemma \ref{lem:x+y} and the induction hypothesis. If $w=w'z_u^{\delta}$, by the induction hypothesis, one calculates
\begin{align*}
\text{L.H.S.} &= (y \diamond_s z_t^{\delta}w)(y \diamond_u 1)+(y^2 \diamond_s z_t^{\delta}w')z_u^{\delta} \\
 &= \left((y \diamond_s 1)z_t^{\delta}w+z_t^{\delta}(y \diamond_t w)\right)(y \diamond_u 1)+\left((y \diamond_s 1)(y \diamond_s z_t^{\delta}w')+z_t^{\delta}(y^2 \diamond_t w')\right)z_u^{\delta}
\end{align*}
and
\begin{align*}
\text{R.H.S.} &= (y \diamond_s 1)\left(z_t^{\delta}w(y \diamond_u 1)+(y \diamond_s z_t^{\delta}w')z_u^{\delta}\right)+z_t^{\delta}\left((y \diamond_t w)(y \diamond_u 1)+(y^2 \diamond_t w')z_u^{\delta}\right), 
\end{align*}
which coincide. If $v=v'z$, it is obvious by Lemma \ref{lem:x+y} and the induction hypothesis. If $v=v'y$, we need to show when $w=1, w'z, w'z_u^{\delta}$ ($w'\in\A_r, u\in\mu_r$). If $w=1$, 
\[
 \text{L.H.S.} 
 =(yv' \diamond_s z_t^{\delta})(y \diamond_t 1)
  +(yv \diamond_s 1)z_t^{\delta}
\]
and
\[
\text{R.H.S.} = (y \diamond_s 1)(v \diamond_s z_t^{\delta})+z_t^{\delta}(yv \diamond_t 1), 
\]
which are equal by Lemma \ref{lem6} (ii) and the induction hypothesis. If $w=w'z$, it is obvious by Lemma \ref{lem:x+y} and the induction hypothesis. If $w=w'z_u^{\delta}$, by using Lemma \ref{lem6} (ii) and the induction hypothesis, one calculates
\begin{align*}
\text{L.H.S.} &= (yv' \diamond_s z_t^{\delta}w)(y \diamond_u 1)+(yv \diamond_s z_t^{\delta}w')z_u^{\delta} \\
&= \left((y \diamond_s 1)(v' \diamond_s z_t^{\delta}w)+z_t^{\delta}(yv' \diamond_t w)\right)(y \diamond_u 1)+\left((y \diamond_s 1)(v \diamond_s z_t^{\delta}w')+z_t^{\delta}(yv \diamond_t w')\right)z_u^{\delta}
\end{align*}
and
\begin{align*}
\text{R.H.S.} &= (y \diamond_s 1)\left((v' \diamond_s z_t^{\delta}w)(y \diamond_u 1)+(v \diamond_s z_t^{\delta}w')z_u^{\delta}\right)+z_t^{\delta}\left((yv' \diamond_t w)(y \diamond_u 1)+(yv \diamond_t w')z_u^{\delta}\right), 
\end{align*}
which coincide. 

\underline{(v):} If $v=1$, by using Lemma \ref{lem:x+y}, the right-hand side turns into
\[
z^2w+zwz-z^2w = zwz, 
\]
which is equal to the left-hand side. If $v=z$, we have
\[
\text{L.H.S.} = (z \diamond_s zw)z = zwz^2
\]
and
\[
\text{R.H.S.} = z(zwz+wz^2-zwz) = zwz^2, 
\]
which coincide. If $v=y$, we need to show when $w=1, w'z, w'z_t^{\delta}$ ($w'\in\A_r, t\in\mu_r$). If $w=1$, 
\[
\text{L.H.S.} = zy \diamond_s z = zyz
\]
and
\[
\text{R.H.S.} = z(y \diamond_s z+zy \diamond_s 1-z(y \diamond_s 1)) = zyz, 
\]
which coincide. If $w=w'z$, by induction on degree of words, the left-hand side turns into
\[
(zy \diamond_s zw')z = z(y \diamond_s zw'+zy \diamond_s w'-z(y \diamond_s w'))z, 
\]
which is equal to the right-hand side due to Lemma \ref{lem:x+y}. If $w=w'z_t^{\delta}$, by using Lemma \ref{lem6} (ii) and the induction hypothesis, one calculates
\[
\text{L.H.S.} = (z \diamond_s zw)(y \diamond_t 1)+(zy \diamond_s zw')z_t^{\delta} = zwz(y \diamond_t 1)+z(y \diamond_s zw'+zy \diamond_s w'-z(y \diamond_s w'))z_t^{\delta}
\]
and
\[
\text{R.H.S.} = z\left(zw(y \diamond_t 1)+(y \diamond_s zw')z_t^{\delta}+wz(y \diamond_t 1)+(zy \diamond_s w')z_t^{\delta}-z(w(y \diamond_t 1)+(y \diamond_s w')z_t^{\delta})\right), 
\]
which are equal. 

If $v=v'z$, it is obvious by Lemma \ref{lem:x+y} and the induction hypothesis. If $v=v'y$, we need to show when $w=1, w'z, w'z_t^{\delta}$ ($w'\in\A_r, t\in\mu_r$). If $w=1$, 
\[
\text{L.H.S.} = zv \diamond_s z = zvz
\]
and
\[
\text{R.H.S.} = z(v \diamond_s z+zv \diamond_s 1-z(v \diamond_s 1)) = zvz, 
\]
which coincide. If $w=w'z$, it is obvious by Lemma \ref{lem:x+y} and the induction hypothesis. If $w=w'z_t^{\delta}$, by using Lemma \ref{lem6} (ii) and the induction hypothesis, one calculates
\begin{align*}
\text{L.H.S.} &= (zv' \diamond_s zw)(y \diamond_t 1)+(zv \diamond_s zw')z_t^{\delta} \\
&= z(v' \diamond_s zw+zv' \diamond_s w-z(v' \diamond_s w))(y \diamond_t 1)+z(v \diamond_s zw'+zv \diamond_s w'-z(v \diamond_s w'))z_t^{\delta}
\end{align*}
and
\begin{align*}
\text{R.H.S.} &= z\left((v' \diamond_s zw)(y \diamond_t 1)+(v \diamond_s zw')z_t^{\delta}+(zv' \diamond_s w)(y \diamond_t 1)\right. \\
&\qquad \left. +(zv \diamond_s w')z_t^{\delta}-z((v' \diamond_s w)(y \diamond_t 1)+(v \diamond_s w')z_t^{\delta})\right), 
\end{align*}
which are equal. 

\underline{(vi):} If $v=1$, by using Lemma \ref{lem:x+y}, the right-hand side turns into
\[
zz_t^{\delta}w+z_t^{\delta}(z \diamond_t w)-zz_t^{\delta}w = z_t^{\delta}wz, 
\]
which is equal to the left-hand side. If $v=z$, we have
\[
\text{L.H.S.} = (z \diamond_s z_t^{\delta}w)z = z_t^{\delta}wz^2
\]
and
\[
\text{R.H.S.} = zz_t^{\delta}wz+z_t^{\delta}wz^2-zz_t^{\delta}wz = z_t^{\delta}wz^2, 
\]
which coincide. If $v=y$, we need to show when $w=1, w'z, w'z_u^{\delta}$ ($w'\in\A_r, u\in\mu_r$). If $w=1$, 
\[
\text{L.H.S.} = zy \diamond_s z_t^{\delta} = z_t^{\delta}z(t \diamond_t 1)+(zy \diamond_s 1)z_t^{\delta}
\]
and
\begin{align*}
\text{R.H.S.} &= z(y \diamond_s z_t^{\delta})+z_t^{\delta}(zy \diamond_t 1)-zz_t^{\delta}(y \diamond_t 1) \\
 &= z(z_t^{\delta}(y \diamond_t 1)+(y \diamond_s 1)z_t^{\delta})+z_t^{\delta}z(y \diamond_t 1)-zz_t^{\delta}(y \diamond_t 1), 
\end{align*}
which coincide. If $w=w'z$, it is obvious by Lemma \ref{lem:x+y} and the induction hypothesis. If $w=w'z_u^{\delta}$, by using Lemma \ref{lem6} (ii) and the induction hypothesis, one calculates
\begin{align*}
\text{L.H.S.} &= (z \diamond_s z_t^{\delta}w)(y \diamond_u 1)+(zy \diamond_s z_t^{\delta}w')z_u^{\delta}\\ 
 &= z_t^{\delta}wz(y \diamond_u 1)+z(y \diamond_s z_t^{\delta}w')z_u^{\delta}+z_t^{\delta}(zy \diamond_t w')z_u^{\delta}-zz_t^{\delta}(y \diamond_s w')z_u^{\delta}
\end{align*}
and
\begin{align*}
\text{R.H.S.} &= z\left(z_t^{\delta}w(y \diamond_u 1)+(y \diamond_s z_t^{\delta}w')z_u^{\delta}\right)+z_t^{\delta}\left(wz(y \diamond_u 1)+(zy \diamond_t w')z_u^{\delta}\right) \\
 &\qquad -zz_t^{\delta}\left(w(y \diamond_u 1)+(y \diamond_t w')z_u^{\delta}\right), 
\end{align*}
which are equal. If $v=v'z$, it is obvious by Lemma \ref{lem:x+y} and the induction hypothesis. If $v=v'y$, we need to show when $w=1, w'z, w'z_u^{\delta}$ ($w'\in\A_r, u\in\mu_r$). If $w=1$, by Lemma \ref{lem6} (ii) and the induction hypothesis, one calculates
\begin{align*}
\text{L.H.S.} &= (zv' \diamond_s z_t^{\delta})(y \diamond_t 1)+(zv \diamond_s 1)z_t^{\delta}\\
 &=\left(z(v' \diamond_s z_t^{\delta})+z_t^{\delta}(zv' \diamond_t 1)-zz_t^{\delta}(v' \diamond_t 1)\right)(y \diamond_t 1)+z(v \diamond_s 1)z_t^{\delta}
\end{align*}
and
\begin{align*}
\text{R.H.S.} &= z(v \diamond_s z_t^{\delta})+z_t^{\delta}(zv \diamond_t 1)-zz_t^{\delta}(v \diamond_t 1)\\
 &= z\left((v' \diamond_s z_t^{\delta})(y \diamond_t 1)+(v \diamond_s 1)z_t^{\delta}\right)+z_t^{\delta}z(v' \diamond_t 1)(y \diamond_t 1)-zz_t^{\delta}(v' \diamond_t 1)(y \diamond_t 1), 
\end{align*}
which are equal. If $w=w'z$, it is obvious by Lemma \ref{lem:x+y} and the induction hypothesis. If $w=w'z_u^{\delta}$, by using Lemma \ref{lem6} (ii) and the induction hypothesis, one calculates
\begin{align*}
\text{L.H.S.} &= (zv' \diamond_s z_t^{\delta}w)(y \diamond_u 1)+(zv \diamond_s z_t^{\delta}w')z_u^{\delta} \\
&= (z(v' \diamond_s z_t^{\delta}w)+z_t^{\delta}(zv' \diamond_t w)-zz_t^{\delta}(v' \diamond_t w))(y \diamond_u 1) \\
&\qquad +(z(v \diamond_s z_t^{\delta}w')+z_t^{\delta}(zv \diamond_t w')-zz_t^{\delta}(v \diamond_t w'))z_u^{\delta}
\end{align*}
and
\begin{align*}
\text{R.H.S.} &= z\left((v' \diamond_s z_t^{\delta}w)(y \diamond_u 1)+(v \diamond_s z_t^{\delta}w')z_u^{\delta}\right)+z_t^{\delta}\left((zv' \diamond_t w)(y \diamond_u 1)+(zv \diamond_t w')z_u^{\delta}\right) \\
&\qquad -zz_t^{\delta}\left((v' \diamond_t w)(y \diamond_u 1)+(v \diamond_t w')z_u^{\delta}\right), 
\end{align*}
which are equal. 

\underline{(vii):} If $v=1$, it is obvious. Otherwise, putting $v=z^{k_1-1}y\cdots z^{k_m-1}y$, one calculates
\[
 \tau (v \diamond_s 1)
 =\tau \left(z^{k_1-1}(z-z_s^{\delta})\cdots z^{k_m-1}(z-z_s^{\delta})\right)
 =z_s^{\delta}z^{k_m-1}\cdots z_s^{\delta}z^{k_1-1}
\]
and
\[
 \tau (v) \diamond_s 1
 =\psi_s(z x^{k_m-1}\cdots z x^{k_1-1})
 =\varphi(z_s x^{k_m-1}\cdots z_s x^{k_1-1}), 
\]
which are equal. 
\end{proof}

%

\section{Proof of Theorem \ref{main1}}\label{pf1}
%
We prove that the polynomial $F_f$ defined just before Proposition \ref{prop7} satisfies the theorem. The proof goes by induction
on $\deg(f)$ for rooted forests $f$ and $\deg(w)$ for words $w$. First, we prove the theorem when $f=\,\begin{xy} {(0,0) \ar @{{*}-{*}} (0,0)} \end{xy}\,$. If $w=1$, we have
\begin{equation*}
 \tilde{f}(z_{s}^{\delta}) 
 %
 =z_{s}^{\delta}(z-z_{s}^{\delta})
\end{equation*}
and
\begin{equation*}
 z_{s}^{\delta}(F_{f}\diamond_{s}1)=z_{s}^{\delta}(y\diamond_{s}1)
  =z_{s}^{\delta}(z-z_{s}^{\delta}), 
\end{equation*}
which are equal. Suppose $\deg(w)\ge1$. If $w=w'z\,(w'\in\mathcal{A}_{r})$, 
by \cite[Theorem 2.2 (d)]{TW22}, which asserts that $R_z$ and any RTM commute, the induction hypothesis, and Lemma \ref{lem:x+y}, we have 
\begin{equation}\label{eq14}
\tilde{f}(z_{s}^{\delta}w'z)=\tilde{f}(z_{s}^{\delta}w')z=z_{s}^{\delta}(F_{f}\diamond_{s}w')z=z_{s}^{\delta}(F_{f}\diamond_{s}w). 
\end{equation}
If $w=w'z_{t}^{\delta}\,(w'\in\mathcal{A}_{r})$,
we have 
\[
\tilde{f}(z_{s}^{\delta}w'z_{t}^{\delta})=\tilde{f}(z_{s}^{\delta}w')z_{t}^{\delta}+z_{s}^{\delta}w'z_{t}^{\delta}(z-z_{t}^{\delta})
\]
and, by Lemma \ref{lem6}, 
\[
z_{s}^{\delta}(y\diamond_{s}w'z_{t}^{\delta}) = z_s^{\delta}(y \diamond_s w')z_{t}^{\delta}+z_{s}^{\delta}w'z_{t}^{\delta}(z-z_{t}^{\delta}), 
\]
which are equal by the induction hypothesis. 

Next, suppose $\deg(f)\ge 2$. If $f=gh\,(g,h\ne\mathbb{I})$, we have
\[
\tilde{f}(z_{s}^{\delta}w)  =\tilde{g}\tilde{h}(z_{s}^{\delta}w)
  =\tilde{g}(z_{s}^{\delta}(F_{h}\diamond_{s}w))
  =z_{s}^{\delta}(F_{g}\diamond_{s}(F_{h}\diamond_{s}w))
  =z_{s}^{\delta}((F_{g}\diamond_{1}F_{h})\diamond_{s}w)
  =z_{s}^{\delta}(F_{f}\diamond_{s}w)
\]
since $\deg(g),\deg(h)<\deg(f)$ and Lemma \ref{lem1}. Let $f$ be a rooted tree and put $f=B_{+}(g)$. When $w=1$,
we have 
\begin{equation}\label{eq13}
\tilde{f}(z_{s}^{\delta}) = R_{z-z_{s}^{\delta}}R_{2z-z_{s}^{\delta}}R_{z-z_{s}^{\delta}}^{-1}\tilde{g}(z_{s}^{\delta}) = R_{z-z_{s}^{\delta}}R_{2z-z_{s}^{\delta}}R_{z-z_{s}^{\delta}}^{-1}z_s^{\delta}(F_g \diamond_s 1) 
\end{equation}
by the induction hypothesis. Since $\psi_s\varphi R_x = R_{z_s^{\delta}}\psi_s\varphi$ and $\psi_s\varphi R_y = R_{z-z_s^{\delta}}\psi_s\varphi$ hold on $\A_1^1$, we have
\begin{equation}\label{eq5}
\eqref{eq13}=z_{s}^{\delta}(\psi_s\varphi(R_{y}R_{x+2y}R_{y}^{-1}(F_{g}))
  =z_{s}^{\delta}(F_{f} \diamond_s 1).
\end{equation}
Suppose $\deg(w)\ge1$. If $w=w'z\,(w'\in\mathcal{A}_r)$, we have \eqref{eq14} again (but this time we consider $\deg(f)\geq 2$). If $w=w'z_{t}^{\delta}\,(w'\in\mathcal{A})$ and $\Delta (f)=\sum_{(f)}f'\otimes f''$, we have 
\[
\tilde{f}(z_{s}^{\delta}w'z_{t}^{\delta}) = \sum_{(f)}\widetilde{f'}(z_{s}^{\delta}w')\widetilde{f''}(z_{t}^{\delta}) = \sum_{(f)}z_{s}^{\delta}(F_{f'}\diamond_{s}w')z_{t}^{\delta}(F_{f''} \diamond_{t} 1)
\]
by the induction hypothesis on degree of words and \eqref{eq5}. This is equal to $z_s^{\delta}(F_f \diamond_s w)$ by Proposition \ref{prop7}. 

Uniqueness of $F_f$ is shown as follows. If $F'_f\in\A_1^1$ also satisfies the theorem, we have
\[
(F_f-F'_f) \diamond_s w =0
\]
for any $s\in\mu_r$ and any $w\in\A_r$. In particular, putting $w=1$ we have
\[
(F_f-F'_f)\diamond_s 1=0,
\]
and hence
\[
F_f-F'_f=\varphi\psi_s^{-1}(0)=0. 
\]
This completes the proof. 

\section{Proof of Theorem \ref{main2}}\label{anti}
For rooted forests $f$, we define polynomials $G_{f}\in\A_{1}^{1}$ recursively by 
\begin{itemize}
\item $G_{\I}=1$,
\item $G_{\,\begin{xy}{(0,0)\ar@{{*}-{*}}(0,0)}\end{xy}\,\,}=-y$, 
\item $G_{t}=L_{2x+y}(G_{f})$ if $t=B_{+}(f)$ and $f\ne\I$, 
\item $G_{f}=G_{g}\diamond_{1}G_{h}$ if $f=gh$, 
\end{itemize}
where $L_v$ denotes the left multiplication map by $v$, i.e., $L_v(w)=vw$ ($v,w\in\A_r$).  The subscript of $G$ is extended linearly. In \cite{MT22}, we find that $G_f=F_{S(f)}$ holds. 
\begin{lem} \label{MT:prop_4.5}
For $f\in\Aug(\calH)$, put $\Delta (f)=\sum_{(f)} f'\otimes f''$. Then we have
\[
 \sum_{(f)} F_{f'} \diamond_1 G_{f''} = 0. 
\]
\end{lem}
\begin{proof}
 See \cite[Proposition 4.5]{MT22}.
\end{proof}

\begin{proof}[Proof of Theorem \ref{main2}]
 If $f=\,\begin{xy}{(0,0)\ar@{{*}-{*}}(0,0)}\end{xy}\,\,$, 
 the theorem holds since $S(f)=-\,\begin{xy}{(0,0)\ar@{{*}-{*}}(0,0)}\end{xy}\,\,$, $G_f=-y$, and Theorem \ref{main1} for $f=\,\begin{xy}{(0,0)\ar@{{*}-{*}}(0,0)}\end{xy}\,\,$. 
 Assume $\deg(f)\ge 2$. 
 If $f=gh\ (g,h\neq\mathbb{I})$, we have
 \[
 \widetilde{S(f)}=\widetilde{S(gh)}=\widetilde{S(h)S(g)}=\widetilde{S(h)}\widetilde{S(g)}=\widetilde{S(g)}\widetilde{S(h)}
 \]
because the antipode $S$ is an anti-automorphism, $\widetilde{\ }$\, is an algebra homomorphism, and RTMs commute with each other. Then, since $\deg(g),\deg(h)<\deg(f)$ and Lemma \ref{lem1}, we have
 \[
  \widetilde{S(f)}(z_s^{\delta}w) 
  = \widetilde{S(g)}(\widetilde{S(h)}(z_s^{\delta}w))
  = z_s^{\delta}(G_g \diamond_s (G_h \diamond_s w)) 
  = z_s^{\delta}((G_g \diamond_1 G_h) \diamond_s w) 
  = z_s^{\delta}(G_f \diamond_s w)). 
 \]
 If $f$ is a tree, by letting $\Delta (f)=\sum_{(f)}f'\otimes f''$ and Lemma \ref{MT:prop_4.5}, we have
 \begin{equation}\label{eq15}
  z_s^{\delta}(G_f \diamond_s w)) 
  =-z_s^{\delta}\sum_{\begin{subarray}{c} (f) \\ f'\neq\mathbb{I} \end{subarray}} 
   (F_{f'} \diamond_1 G_{f''}) \diamond_s w.
 \end{equation}
By Lemma \ref{lem1}, Theorem \ref{main1}, and the induction hypothesis, we have
\[
\eqref{eq15}=-z_s^{\delta}\sum_{\begin{subarray}{c} (f) \\ f'\neq\mathbb{I} \end{subarray}} 
   F_{f'} \diamond_s (G_{f''} \diamond_s w)
   =-\sum_{\begin{subarray}{c} (f) \\ f'\neq\mathbb{I} \end{subarray}}  
   \widetilde{f'}(z_s^{\delta}(G_{f''} \diamond_s w))
   =-\sum_{\begin{subarray}{c} (f) \\ f'\neq\mathbb{I} \end{subarray}} 
   \widetilde{f'}(\widetilde{S(f'')}(z_s^{\delta}w)).
\]
Since $\sum_{(f)}f'S(f'')=0$, we get the theorem. 
\end{proof}

\section{Proof of Theorem \ref{main3}} \label{pf2}
\begin{lem} \label{MT:prop_5.1}
For $f\in\Aug(\calH)$, we have
\[
F_f=-R_y\tau R_y^{-1}(F_{S(f)}). 
\]
\end{lem}
\begin{proof}
 See \cite[Proposition 5.1]{MT22}.
\end{proof}
\begin{proof}[Proof of Theorem \ref{main3}]
First, we prove the theorem when $w=z_{s}^{\delta}w'(z-z_t^{\delta})\in z_{s}^{\delta}\A_r(z-z_t^{\delta})$. By Theorem \ref{main2}, we have 
 \[
  \widetilde{S(f)}(w)=z_{s}^{\delta}(F_{S(f)}\diamond_{s}w'(z-z_t^{\delta})).
 \]
 We also have 
 \begin{equation}\label{eq16}
  \tau\tilde{f}\tau(w) 
  =\tau\tilde{f}(z_{t}^{\delta}\tau(w')(z-z_s^{\delta}))
  =\tau(z_{t}^{\delta}(F_{f}\diamond_{t}\tau(w')(z-z_s^{\delta})))
 \end{equation}
by Theorem \ref{main1}. Then, by Lemma \ref{MT:prop_5.1}, we have
\[
\eqref{eq16} = -\tau(z_{t}^{\delta}
   (R_y\tau R_{y}^{-1}(F_{S(f)})\diamond_{t}\tau(w')(z-z_s^{\delta}))), 
\]
which is equal to $z_{s}^{\delta}(F_{S(f)}\diamond_{s}w'(z-z_t^{\delta}))$ because of Proposition \ref{prop:main3}. 

Next, we consider when $w=w'z\in \A_r z$. 
Since $R_z$ and RTMs commute, we have
\[
\widetilde{S(f)}(w) =\widetilde{S(f)}(xw')z
\]
and
\[
\tau\tilde{f}\tau(w) =\tau\tilde{f}\tau(xw'z)=\tau\tilde{f}\tau(xw')z, 
\]
which are equal by the induction hypothesis. Similarly, since $L_z$ and RTMs commute, we have the same consequence when $w=zw'\in z\A_r$. 
\end{proof}

\section*{Acknowledgement}
This work is supported by JSPS KAKENHI Grant Numbers JP19K03434 and JP22K13897, Grant for Basic Science Research Projects
from Sumitomo Foundation, and Research Funding Granted by
the University of Kitakyushu.



\begin{thebibliography}{100}

\bibitem{AK04}
T. Arakawa and M. Kaneko, 
{\em On multiple $L$-values}, 
J. Math.\ Soc.\ Japan \textbf{56} (2004), 967--991. 




\bibitem{CK98}
A. Connes and D. Kreimer, 
{\em Hopf algebras, Renormalization and Noncommutative Geometry}, 
Commun.\ Math.\ Phys.\ \textbf{199} (1998), 203--242. 


\bibitem{HMO19}
M. Hirose, H. Murahara, and T. Onozuka, 
{\em $\mathbb{Q}$-linear relations of specific families of multiple zeta values and the linear part of Kawashima's relation}, 
manuscripta math.\ \textbf{164} (2021), 455-465. 






\bibitem{KT08}
G. Kawashima and T. Tanaka,
{\em Newton series and extended derivation relations for multiple $L$-values}, 
arXiv:0801.3062.



\bibitem{MT22}
H. Murahara and T. Tanaka, 
{\em Algebraic aspects of rooted tree maps},
Ramanujan J. (to appear).


\bibitem{Tan19} 
T. Tanaka,
{\em Rooted tree maps}, 
Commun.\ Num. Th. Phys.\ \textbf{13} (2019), 647--666.

\bibitem{TW22} 
T. Tanaka and N. Wakabayashi,
{\em Rooted tree maps for multiple $L$-values}, 
J. Number Theory {\bf 240} (2022), 471--489. 

\end{thebibliography}
\end{document}